\documentclass[final]{siamltex}

\usepackage{amssymb,amsmath,amsfonts}
\usepackage{graphicx}
\usepackage{hyperref}

\newtheorem{assumption}[theorem]{Assumption}
\newtheorem{lem}[theorem]{Lemma}
\newtheorem{cor}[theorem]{Corollary}
\newtheorem{defn}[theorem]{Definition}
\newtheorem{remark}[theorem]{Remark}

\newcommand{\abs}[1]{| #1 |}
\DeclareMathOperator*{\argmax}{arg\,max}
\DeclareMathOperator*{\argmin}{arg\,min}
\newcommand{\bbE}{\mathbb{E}}
\newcommand{\bbK}{\mathbb{M}}
\newcommand{\bbN}{\mathbb{N}}
\newcommand{\bbP}{\mathbb{P}}
\newcommand{\bbQ}{\mathbb{Q}}
\newcommand{\bbR}{\mathbb{R}}
\newcommand{\bz}{\bar z}
\newcommand{\cC}{\mathcal{C}}
\newcommand{\cG}{\mathcal{G}}
\newcommand{\cH}{\mathcal{H}}
\newcommand{\cN}{\mathcal{N}}
\newcommand{\e}{\mathrm{e}}
\let\eps=\varepsilon
\newcommand{\gd}{\delta}
\newcommand{\la}{\langle}
\newcommand{\N}{\mathbb{N}}
\newcommand{\norm}[1]{\| #1 \|}
\newcommand{\pd}{\partial}
\newcommand{\quark}{{\setbox0\hbox{$x$}\hbox to\wd0{\hss$\cdot$\hss}}}
\newcommand{\R}{\mathbb{R}}
\newcommand{\ra}{\rangle}
\newcommand{\rd}{\mathrm{d}}
\newcommand{\TT}{\mathbb{T}}
\newcommand{\ud}{\mathrm{d}}
\newcommand{\utr}{u^\dagger}
\newcommand{\Z}{\mathbb{Z}}

\begin{document}

\title{MAP Estimators and Their Consistency in Bayesian Nonparametric Inverse Problems}
\author{M. Dashti\footnote{Department of Mathematics, University of Sussex, Brighton BN1 9QH, UK},
  K.J.H. Law,
  A.M. Stuart\footnote{Mathematics Institute, University of Warwick, Coventry, CV4 7AL, UK} and
  J. Voss\footnote{School of Mathematics, University of Leeds, Leeds, LS2 9JT, UK}}
\maketitle

\begin{abstract}
  We consider the inverse problem of estimating an unknown function
  $u$ from noisy measurements $y$ of a known, possibly nonlinear, map
  $\cG$ applied to~$u$.  We adopt a Bayesian approach to the problem
  and work in a setting where the prior measure is specified as a
  Gaussian random field~$\mu_0$. We work under a natural set of
  conditions on the likelihood which imply the existence of a
  well-posed posterior measure, $\mu^y$.  Under these conditions we
  show that the {\em maximum a posteriori} (MAP) estimator is
  well-defined as the minimiser of an Onsager-Machlup functional
  defined on the Cameron-Martin space of the prior; thus we link a
  problem in probability with a problem in the calculus of variations.
  We then consider the case where the observational noise vanishes and
  establish a form of Bayesian posterior consistency for the MAP estimator.
We also prove a
  similar result for the case where the observation of $\cG(u)$ can be
  repeated as many times as desired with independent identically
  distributed noise.  The theory is illustrated with examples from an
  inverse problem for the Navier-Stokes equation, motivated by
  problems arising in weather forecasting, and from the theory of
  conditioned diffusions, motivated by problems arising in molecular
  dynamics.
\end{abstract}


\section{Introduction}

This article considers questions from Bayesian statistics in an
infinite dimensional setting, for example in function spaces.  We
assume our state space to be a general separable Banach space
$\bigl(X, \|\quark\|_X\bigr)$.  While in the finite-dimensional
setting, the prior and posterior distribution of such statistical
problems can typically be described by densities w.r.t.\ the Lebesgue
measure, such a characterisation is no longer possible in the infinite
dimensional spaces we consider here: it can be shown that no analogue
of the Lebesgue measure exists in infinite dimensional spaces.  One
way to work around this technical problem is to replace Lebesgue
measure with a Gaussian measure on $X$, \textit{i.e.}\ with a Borel
probability measure~$\mu_0$ on $X$ such that all finite-dimensional
marginals of $\mu_0$ are (possibly degenerate) normal
distributions.  Using a fixed, centred (mean-zero) Gaussian measure
$\mu_0 = \cN(0,\cC_0)$ as a reference measure, we then assume that the
distribution of interest, $\mu$, has a density with respect to $\mu_0$:
\begin{equation}\label{e:radon}
  \frac{\ud \mu}{\ud \mu_0}(u)
  \propto \exp\bigl(-\Phi(u) \bigr).
\end{equation}
Measures $\mu$ of this form arise naturally in a number of
applications, including the theory of conditioned diffusions
\cite{HSV11} and the Bayesian approach to inverse
problems~\cite{St10}.  In these settings there are many applications
where $\Phi\colon X\to\R$ is a locally Lipschitz continuous function
and it is in this setting that we work.

Our interest is in defining the concept of ``most likely'' functions
with respect to the measure $\mu$, and in particular the {\em maximum
  a posteriori} estimator in the Bayesian context. We will refer to
such functions as MAP estimators throughout. We will define the
concept precisely and link it to a problem in the calculus of
variations, study posterior consistency of the MAP estimator in the
Bayesian setting, and compute it for a number of illustrative
applications.

To motivate the form of MAP estimators considered here we consider the
case where $X=\R^d$ is finite dimensional and the prior $\mu_0$ is
Gaussian $\cN(0,\cC_0)$.  This prior has density
$\exp(-\frac12|\cC_0^{-1/2}u|^2)$ with respect to the Lebesgue measure
where $|\quark|$ denotes the Euclidean norm.  The probability density
for $\mu$ with respect to the Lebesgue measure, given
by~\eqref{e:radon}, is maximised at minimisers of
\begin{equation}\label{eq:I-finite}
  I(u):=\Phi(u)+\frac12\|u\|_E^2
\end{equation}
where $\|\quark\|_E=|\cC_0^{-1/2}u|$.  We would like to derive such a
result in the infinite dimensional setting.

The natural way to talk about MAP estimators in the infinite
dimensional setting is to seek the centre of a small ball with maximal
probability, and then study the limit of this centre as the radius of
the ball shrinks to zero.  To this end, let $B^\delta(z)\subset X$ be
the open ball of radius $\delta$ centred at $z\in X$. If there is a
functional $I$, defined on $E$, which satisfies
\begin{equation}\label{eq:OM1}
  \lim_{\delta\to 0}\frac{\mu(B^\delta(z_2))}{\mu(B^\delta(z_1))}
  = \exp\bigl(I(z_1)-I(z_2)\bigr),
\end{equation}
then $I$ is termed the {\em Onsager-Machlup} functional
\cite{db78,iw89}.  For any fixed $z_1$, the function $z_2$ for which
the above limit is maximal is a natural candidate for the MAP
estimator of $\mu$ and is clearly given by minimisers of the
Onsager-Machlup function. In the finite dimensional case it is clear
that $I$ given by \eqref{eq:I-finite} is the Onsager-Machlup
functional.

From the theory of infinite dimensional Gaussian measures
\cite{Lif95,Bog98} it is known that copies of the
Gaussian measure $\mu_0$ shifted by $z$ are absolutely continuous w.r.t.\ $\mu_0$
itself, if and only if $z$ lies in the Cameron-Martin space $\bigl(E,
\langle \cdot, \cdot \rangle_E, \|\quark\|_E\bigr)$; furthermore, if the shift
direction $z$ is in $E$, then shifted measure~$\mu_z$ has density
\begin{equation}\label{eq:CM-intro}
  \frac{d\mu_z}{d\mu_0}
  = \exp\Bigl(\la z, u\ra_E - \frac12 \|z\|_E^2\Bigr).
\end{equation}
In the finite dimensional example, above, the Cameron-Martin norm of
the Gaussian measure $\mu_0$ is the norm~$\|\quark\|_E$ and it is easy
to verify that~\eqref{eq:CM-intro} holds for all $z\in\R^d$.  In the
infinite dimensional case, it is important to keep in mind
that~\eqref{eq:CM-intro} only holds for $z\in E\subsetneq X$.
Similarly, the relation~\eqref{eq:OM1} only holds for $z_1, z_2\in E$.
In our application, the Cameron-Martin formula~\eqref{eq:CM-intro} is
used to bound the probability of the shifted ball $B^\delta(z_2)$ from
equation~\eqref{eq:OM1}.  (For an exposition of the standard results
about small ball probabilities for Gaussian measures we refer
to~\cite{Bog98,Lif95}; see also \cite{LedTal91} for related material.)
The main technical difficulty that is
encountered stems from the fact that the Cameron-Martin space~$E$,
while being dense in~$X$, has measure zero with respect to~$\mu_0$.
An example where this problem can be explicitly seen is the case where
$\mu_0$ is the Wiener measure on $L^2$; in this example $E$
corresponds to a subset of the Sobolov space~$H^1$, which has indeed measure zero
w.r.t.\ Wiener measure.

Our theoretical results assert that despite these technical
complications the situation from the finite-dimensional example,
above, carry over to the infinite dimensional case essentially without
change.  In Theorem~\ref{t:OM} we show that the Onsager-Machlup
functional in the infinite dimensional setting still has the
form~\eqref{eq:I-finite}, where $\|\quark\|_E$ is now the
Cameron-Martin norm associated to~$\mu$ (using $\| z \|_E = \infty$
for $z\in X\setminus E$), and in Corollary~\ref{c:MAPmin} we show that
the MAP estimators for $\mu$ lie in the Cameron-Martin space~$E$ and
coincide with the minimisers of the Onsager-Machlup functional~$I$.

In the second part of the paper, we consider the inverse problem of
estimating an unknown function $u$ in a Banach space $X$, from a given
observation $y \in \bbR^J$, where
\begin{equation}\label{e:obs}
  y=G(u)+\zeta;
\end{equation}
here $G\colon X\to \bbR^J$ is a possibly nonlinear operator, and
$\zeta$ is a realization of an $\bbR^J$-valued centred Gaussian random
variable with known covariance $\Sigma$.  A prior probability measure
$\mu_0(\ud u)$ is put on $u$, and the distribution of $y|u$ is given
by \eqref{e:obs}, with $\zeta$ assumed independent of $u$.  Under
appropriate conditions on $\mu_0$ and $G$, Bayes theorem is
interpreted as giving the following formula for the Radon-Nikodym
derivative of the posterior distribution $\mu^y$ on $u|y$ with respect
to $\mu_0$:
\begin{equation}\label{e:radon1}
  \frac{\ud\mu^y}{\ud\mu_0}(u)
  \propto\exp \bigl(-\Phi(u;y)\bigr),
\end{equation}
where
\begin{equation}\label{eq:fy}
  \Phi(u;y)
  = \frac12\Bigl|\Sigma^{-\frac12}\bigl(y-G(u)\bigr)\Bigr|^2.
\end{equation}
Derivation of Bayes formula \eqref{e:radon1} for problems with finite
dimensional data, and $\zeta$ in this form, is discussed in
\cite{CDRS09}.  Clearly, then, Bayesian inverse problems with Gaussian
priors fall into the class of problems studied in this paper, for
potentials $\Phi$ given by \eqref{eq:fy} which depend on the observed
data $y$.  When the probability measure $\mu$ arises from the Bayesian
formulation of inverse problems, it is natural to ask whether the MAP
estimator is close to the truth underlying the data, in either the
small noise or large sample size limits. This is a form of Bayesian
posterior consistency, here defined in terms of the MAP estimator
only. We will study this question for finite observations of a
nonlinear forward model, subject to Gaussian additive noise.

The paper is organized as follows:
\begin{itemize}
\item in section~\ref{s:Bayes} we detail our assumptions on $\Phi$ and
  $\mu_0$;
\item in section~\ref{s:MAP} we give conditions for the existence of
  an Onsager-Machlup functional $I$ and show that the MAP estimator is
  well-defined as the minimiser of this functional;
\item in section~\ref{s:consistency} we study the problem of Bayesian
  posterior consistency by studying limits of Onsager-Machlup
  minimisers in the small noise and large sample size limits;
\item in section~\ref{sec:fm} we study applications arising
from data assimilation for the Navier-Stokes equation, as
a model for what is done in weather prediction;
\item in section~\ref{sec:cd} we study applications arising in the
  theory of conditioned diffusions.
\end{itemize}

We conclude the introduction with a brief literature review.  We first
note that MAP estimators are widely used in practice in the infinite
dimensional context \cite{Rasmussen2005,Kaipio2005}.  We also note
that the functional $I$ in \eqref{eq:I-finite} resembles a
Tikhonov-Phillips regularization of the minimisation problem for
$\Phi$ \cite{Eng96}, with the Cameron-Martin norm of the prior
determining the regularization. In the theory of classical
non-statistical inversion, formulation via Tikhonov-Phillips
regularization leads to an infinite dimensional optimization problem
and has led to deeper understanding and improved algorithms.  Our aim
is to achieve the same in a probabilistic context.  One way of
defining a MAP estimator for $\mu$ given by \eqref{e:radon} is to
consider the limit of parametric MAP estimators: first discretize the
function space using $n$ parameters, and then apply the finite
dimensional argument above to identify an Onsager-Machlup functional
on $\bbR^n$. Passing to the limit $n \to \infty$ in the functional
provides a candidate for the limiting Onsager-Machlup functional.
This approach is taken in \cite{mar09,mor68,rog10} for problems
arising in conditioned diffusions.  Unfortunately, however, it does
not necessarily lead to the correct identification of the
Onsager-Machlup functional as defined by \eqref{eq:OM1}.  The reason
for this is that the space on which the Onsager-Mahlup functional is
defined is smoother than the space on which small ball probabilities
are defined. Small ball probabilities are needed to properly define
the Onsager-Machlup functional in the infinite dimensional limit. This
means that discretization and use of standard numerical analysis limit
theorems can, if incorrectly applied, use more regularity than is
admissible in identifying the limiting Onsager-Mahlup functional.  We
study the problem directly in the infinite dimensional setting,
without using discretization, leading, we believe, to greater
clarity. Adopting the infinite dimensional perspective for MAP
estimation has been widely studied for diffusion processes
\cite{DeZe91} and related stochastic PDEs \cite{Zet89}; see
\cite{Zet00} for an overview. Our general setting is similar to that
used to study the specific applications arising in the papers
\cite{DeZe91,Zet89,Zet00}.  By working with small ball properties of
Gaussian measures, and assuming that $\Phi$ has natural continuity
properties, we are able to derive results in considerable generality.
There is a recent related definition of MAP estimators in
\cite{Heg07}, with application to density estimation in
\cite{GrHeg10}. However, whilst the goal of minimising $I$ is also
identified in \cite{Heg07}, the proof in that paper is only valid in
finite dimensions since it implicitly assumes that the Cameron-Martin
norm is $\mu_0-$a.s. finite.  In our specific application to fluid
mechanics our analysis demonstrates that widely used {\em variational
  methods} \cite{ben} may be interpreted as MAP estimators for an
appropriate Bayesian inverse problem and, in particular, that this
interpretation, which is understood in the atmospheric sciences
community in the finite dimensional context, is well-defined in the
limit of infinite spatial resolution.

Posterior consistency in Bayesian nonparametric statistics has a long
history \cite{ghosal99}.  The study of posterior consistency for the
Bayesian approach to inverse problems is starting to receive
considerable attention. The papers \cite{bart11,als12} are devoted to
obtaining rates of convergence for linear inverse problems with
conjugate Gaussian priors, whilst the papers \cite{BocGr11,Ray12}
study non-conjugate priors for linear inverse problems.  Our analysis
of posterior consistency concerns nonlinear problems, and finite data
sets, so that multiple solutions are possible. We prove an appropriate
weak form of posterior consistency, without rates, building on ideas
appearing in \cite{BisH04}.

Our form of posterior consistency is weaker than the general form of
Bayesian posterior consistency since it does not concern fluctuations
in the posterior, simply a point (MAP) estimator. However we note that
for linear Gaussian problems there are examples where the conditions
which ensure convergence of the posterior mean (which coincides with
the MAP estimator in the linear Gaussian case) also ensure posterior
contraction of the entire measure \cite{als12,bart11}.


\section{Set-up}
\label{s:Bayes}

Throughout this paper we assume that $\bigl( X, \|\quark\|_X\bigr)$ is
a separable Banach space and that $\mu_0$ is a centred Gaussian
(probability) measure on $X$ with Cameron-Martin space $\bigl(E,
\langle \cdot, \cdot \rangle_E, \|\quark\|_E\bigr)$.  The measure
$\mu$ of interest is given by~\eqref{e:radon} and we make the
following assumptions concerning the {\em potential}~$\Phi$.

\begin{assumption} \label{a:asp1}
  The function $\Phi\colon X\to\R$ satisfies the following conditions:
  \begin{itemize}
  \item[(i)] For every $\eps>0$ there is an $M\in \R$, such that
    for all $u\in X$,
    \begin{equation*}
      \Phi(u) \geq M -\eps\|u\|_X^2.
    \end{equation*}
  \item[(ii)] $\Phi$ is locally bounded from above, \textit{i.e.}\ for
    every $r>0$ there exists $K=K(r)>0$ such that, for all $u\in X$
    with $\|u\|_X<r$ we have
    \begin{equation*}
      \Phi(u) \leq K.
    \end{equation*}
  \item[(iii)] $\Phi$ is locally Lipschitz continuous, \textit{i.e.}\
    for every $r>0$ there exists $L=L(r)>0$ such that for all
    $u_1,u_2\in X$ with $\|u_1\|_X,\|u_2\|_X < r$ we have
    \begin{equation*}
      |\Phi(u_1)-\Phi(u_2)| \leq L\|u_1-u_2\|_X.
    \end{equation*}
  \end{itemize}
\end{assumption}

Assumption~\ref{a:asp1}(i) ensures that the expression~\eqref{e:radon}
for the measure $\mu$ is indeed normalizable to give a probability
measure; the specific form of the lower bound is designed to ensure
that application of the Fernique Theorem (see \cite{Bog98}
or~\cite{Lif95}) proves that the required normalization constant is
finite.  Assumption~\ref{a:asp1}(ii) enables us to get explicit bounds
from below on small ball probabilities and
Assumption~\ref{a:asp1}(iii) allows us to use continuity to control
the Onsager-Machlup functional.  Numerous examples satisfying these
condition are given in the references~\cite{St10,HSV11}.  Finally, we
define a function~$I\colon X\to \R$ by
\begin{equation}\label{eq:I}
  I(u) =
  \begin{cases}
    \Phi(u) + \frac12\|u\|_E^2 & \mbox{if $u \in E$, and} \\
    +\infty & \mbox{else.}
  \end{cases}
\end{equation}
We will see in section~\ref{s:MAP} that $I$ is the Onsager-Machlup
functional.

\begin{remark}\label{r:mean}
  We close with a brief remark concerning the definition of the
  Onsager-Machlup function in the case of non-centred reference
  measure $\mu_0=\cN(m,\cC_0)$.  Shifting coordinates by $m$ it is
  possible to apply the theory based on centred Gaussian measure
  $\mu_0$, and then undo the coordinate change.  The relevant
  Onsager-Machlup functional can then be shown to be
  \begin{equation*}
    I(u) =
    \begin{cases}
      \Phi(u) + \frac12\|u-m\|_E^2 & \mbox{if $u-m \in E$, and} \\
      +\infty & \mbox{else.}
    \end{cases}
  \end{equation*}
\end{remark}


\section{MAP estimators and the Onsager-Machlup functional}
\label{s:MAP}

In this section we prove two main results. The first, Theorem
\ref{t:OM}, establishes that $I$ given by \eqref{eq:I-finite} is
indeed the Onsager-Machlup functional for the measure $\mu$ given by
\eqref{e:radon}. Then Theorem~\ref{t:MAP} and Corollary
\ref{c:MAPmin}, show that the MAP estimators, defined precisely in
Definition~\ref{d:MAP}, are characterised by the minimisers of the
Onsager-Machlup functional.

For $z \in X$, let $B^\delta(z)\subset X$ be the open ball centred at
$z\in X$ with radius~$\delta$ in~$X$.  Let
\begin{equation*}
  J^{\delta}(z)
  = \mu\bigl( B^\delta(z) \bigr)
\end{equation*}
be the mass of the ball~$B^\delta(z)$.  We first define the
MAP estimator for $\mu$ as follows:

\begin{defn}\label{d:MAP}
  Let
  \begin{equation*}
    z^\delta=\argmax_{z \in X} J^{\delta}(z).
  \end{equation*}
  Any point $\tilde z\in X$ satisfying $\lim_{\gd\to 0}(J^\gd(\tilde
  z)/J^\gd(z^\gd))=1$, is a MAP estimator for the measure $\mu$ given
  by~\eqref{e:radon}.
\end{defn}

We show later on (Theorem~\ref{t:MAP}) that a strongly convergent
subsequence of $\{z^\delta\}_{\delta>0}$ exists and its limit, that we
prove to be in $E$, is a MAP estimator and also minimises the
Onsager-Machlup functional $I$.  Corollary~\ref{c:MAPmin} then shows
that any MAP estimator $\tilde z$ as given in Definition~\ref{d:MAP}
lives in $E$ as well, and minimisers of $I$ characterise all MAP
estimators of $\mu$.

One special case where it is easy to see that the MAP estimator is
unique is the case where $\Phi$ is linear, but we note that, in
general, the MAP estimator cannot be expected to be unique.  To
achieve uniqueness, stronger conditions on $\Phi$ would be required.

We first need to show that $I$ is the Onsager-Machlup
functional for our problem:

\begin{theorem}\label{t:OM}
  Let Assumption~\ref{a:asp1} hold.  Then the function $I$ defined
  by~\eqref{eq:I} is the Onsager-Machlup functional for~$\mu$,
  \textit{i.e.}\ for any $z_1,z_2\in E$ we have
  \begin{equation*}
    \lim_{\delta\to 0}\frac{J^\delta (z_1)}{J^\delta(z_2)}
    =\exp\bigl(I(z_2)-I(z_1)\bigr).
  \end{equation*}
\end{theorem}

\begin{proof}
  Note that $J^\delta(z)$ is finite and positive for any $z \in E$ by
  Assumptions~\ref{a:asp1}(i),(ii) together with the Fernique Theorem
  and the positive mass of all balls in $X$, centred at points in $E$,
  under Gaussian measure \cite{Bog98}.  The key estimate in the proof
  is the following consequence of Proposition~3 in Section~18
  of~\cite{Lif95}:
  \begin{equation}\label{eq:need}
    \lim_{\delta\to 0}
    \frac{\mu_0\bigl( B^\delta(z_1) \bigr)}
    {\mu_0\bigl(B^\delta(z_2)\bigr)}
    = \exp\left( \frac12\|z_2\|_E^2-\frac12\|z_1\|_E^2 \right).
  \end{equation}
  This is the key estimate in the proof since it transfers questions
  about probability, naturally asked on the space $X$ of full measure
  under $\mu_0$, into statements concerning the Cameron-Martin norm of
  $\mu_0$, which is almost surely infinite under $\mu_0$.

  We have
  \begin{align*}
    \frac{J^\delta(z_1)}{J^\delta(z_2)}
    &=\frac{ \int_{B^\delta(z_1)}\exp (-\Phi(u))\,\mu_0(\ud u) }
    { \int_{B^\delta(z_2)}\exp (-\Phi(v))\,\mu_0(\ud v) } \\
    &=\frac{ \int_{B^\delta(z_1)}\exp (-\Phi(u)+\Phi(z_1))\exp(-\Phi(z_1))\,\mu_0(\ud u) }
    { \int_{B^\delta(z_2)}\exp (-\Phi(v)+\Phi(z_2))\exp (-\Phi(z_2))\,\mu_0(\ud v) }.
  \end{align*}
  By Assumption~\ref{a:asp1} (iii), for any $u,v\in X$
  \begin{equation*}
    -L\,\|u-v\|_X\,\le\,\Phi(u)-\Phi(v)\,\le\,L\,\|u-v\|_X
  \end{equation*}
  where $L=L(r)$ with $r>\max\{\|u\|_X,\|v\|_X\}$.  Therefore, setting
  $L_1=L(\|z_1\|_X+\delta)$ and $L_2=L(\|z_2\|_X+\delta)$, we can
  write
  \begin{align*}
    \frac{J^\delta(z_1)}{J^\delta(z_2)}
    &\le \e^{\delta(L_1+L_2)}\frac{ \int_{B^\delta(z_1)}\exp (-\Phi(z_1))\,\mu_0(\ud u) }
    { \int_{B^\delta(z_2)}\exp (-\Phi(z_2))\,\mu_0(\ud v) } \\
    &=\e^{\delta(L_1+L_2)}\e^{-\Phi(z_1)+\Phi(z_2)}
    \frac{  \int_{B^\delta(z_1)}\,\mu_0(\ud u) }
    { \int_{B^\delta(z_2)}\,\mu_0(\ud v) }.
  \end{align*}
  Now, by \eqref{eq:need}, we have
  \begin{equation*}
    \frac{J^\delta(z_1)}{J^\delta(z_2)}
    \le r_1(\delta)\,\e^{\delta (L_2+L_1)}\e^{-I(z_1)+I(z_2)}
  \end{equation*}
  with $r_1(\delta)\to 1$ as $\delta\to 0$. Thus
  \begin{equation}\label{e:Jlims}
    \limsup_{\delta\to 0}\frac{J^\delta(z_1)}{J^\delta(z_2)}
    \,\le\, \e^{-I(z_1)+I(z_2)}
  \end{equation}
  Similarly we obtain
  \begin{equation*}
    \frac{J^\delta(z_1)}{J^\delta(z_2)}
    \ge \frac{1}{r_2(\delta)}\,\e^{-\delta (L_2+L_1)}\e^{-I(z_1)+I(z_2)}
  \end{equation*}
   with $r_2(\delta)\to 1$ as $\delta\to 0$ and deduce that
  \begin {equation}\label{e:Jlimi}
    \liminf_{\delta\to 0}\frac{J^\delta(z_1)}{J^\delta(z_2)}
    \,\ge\, \e^{-I(z_1)+I(z_2)}
  \end{equation}
  Inequalities \eqref{e:Jlims} and \eqref{e:Jlimi} give the desired
  result.
\end{proof}

We note that similar methods of analysis show the following:

\begin{cor}
  Let the Assumptions of Theorem~\ref{t:OM} hold.  Then for any $z\in
  E$
  \begin{equation*}
    \lim_{\delta\to 0}\frac{J^\delta(z)}{\int_{B^{\delta}(0)}\mu_0(\ud u)}
    = \frac{1}{Z}\,\e^{-I(z)},
  \end{equation*}
  where $Z=\int_X \exp(-\Phi(u))\,\mu_0(\ud u)$.
\end{cor}

\begin{proof}
  Noting that we consider $\mu$ to be a probability measure and hence
  \begin{equation*}
    \frac{J^\delta(z)}{\int_{B^{\delta}(0)}\mu_0(\ud u)}
    =\frac{\frac{1}{Z}\int_{B^{\delta}(z)}\exp(-\Phi(u))\mu_0(\ud u)}{\int_{B^{\delta}(0)}\mu_0(\ud u)},
  \end{equation*}
  with $Z=\int_X \exp(-\Phi(u))\,\mu_0(\ud u)$, arguing along the
  lines of the proof of the above theorem gives
  \begin{equation*}
    \frac{1}{Z}\frac{1}{r(\delta)}\e^{-\delta \hat L}\e^{-I(z)}\le\frac{J^\delta(z)}{\int_{B^{\delta}(0)}\mu_0(\ud u)}\le \frac{1}{Z}r(\delta)\e^{\delta \hat L}\e^{-I(z)}
  \end{equation*}
  with $\hat L=L(\|z\|_X+\delta)$ (where $L(\cdot)$ is as in
  Definition~\ref{a:asp1}) and $r(\delta) \to 1$ as $\delta \to 0$.
  The result then follows by taking $\lim\sup$ and $\lim\inf$ as
  $\delta\to 0$.
\end{proof}

\begin{proposition}\label{p:prop}
  Suppose Assumptions~\ref{a:asp1} hold. Then the minimum of $I\colon
  E \to \bbR$ is attained for some element $z^{*} \in E$.
\end{proposition}

\begin{proof}
  The existence of a minimiser of $I$ in $E$, under the given
  assumptions, is proved as Theorem 5.4 in \cite{St10} (and as Theorem
  2.7 in \cite{CDRS09} in the case that $\Phi$ is non-negative).
\end{proof}

The rest of this section is devoted to a proof of the result that MAP
estimators can be characterised as minimisers of the Onsager-Machlup
functional~$I$ (Theorem~\ref{t:MAP} and Corollary~\ref{c:MAPmin}).

\begin{theorem}\label{t:MAP}
  Suppose that Assumptions~\ref{a:asp1} (ii) and (iii) hold. Assume
  also that there exists an $M\in\R$ such that $\Phi(u)\ge M$ for any
  $u\in X$.
  \begin{itemize}
  \item[i)] Let $z^\delta=\argmax_{z\in X}J^\delta(z)$.  There is a
    $\bar z\in E$ and a subsequence of $\{z^{\delta}\}_{\delta>0}$
    which converges to $\bz$ strongly in $X$.
  \item[ii)] The limit $\bz$ is a MAP estimator and a minimiser of $I$.
  \end{itemize}
\end{theorem}

The proof of this theorem is based on several lemmas.  We state and
prove these lemmas first and defer the proof of Theorem~\ref{t:MAP} to
the end of the section where we also state and prove a corollary
characterising the MAP estimators as minimisers of Onsager-Machlup
functional.

\begin{lemma}\label{l:limg0}
  Let $\delta>0$. For any centred Gaussian measure $\mu_0$ on a
  separable Banach space $X$ we have
  \begin{equation*}
    \frac{J^\delta_0(z)}{J^\delta_0(0)}\le c\,\e^{-\frac{a_1}{2}(\|z\|_X-\delta)^2},
  \end{equation*}
  where $c=\exp(\frac{a_1}{2}\gd^2)$ and $a_1$ is a constant
  independent of $z$ and $\gd$.
\end{lemma}

\begin{proof}
  We first show that this is true for a centred Gaussian measure on
  $\R^n$ with the covariance matrix
  $C=\diag[\lambda_1, \ldots,\lambda_n]$ in basis $\{e_1, \ldots,e_n\}$,
  where $\lambda_1\ge \lambda_2\ge\dots\ge \lambda_n$.  Let
  $a_j=1/\lambda_j$, and $|z|^2=z_1^2+\cdots+z_n^2$.  Define
  \begin{equation}\label{e:defJ0n}
    J^\gd_{0,n}(z)
    := \int_{B^\gd(z)} \e^{-\frac{1}{2}(a_1x_1^2+\cdots+a_nx_n^2)}\,\ud x,
                    \quad\mbox{for any } z\in\R^n,
  \end{equation}
  and with $B^\gd(z)$ the ball of radius $\gd$ and centre $z$ in
  $\R^n$.  We have
  \begin{align*}
  \frac{J^\gd_{0,n}(z)}{J^\gd_{0,n}(0)}
  &=\frac{ \int_{B^\gd(z)} \e^{-\frac{1}{2}(a_1x_1^2+\cdots+a_nx_n^2)}\,\ud x }{ \int_{B^\gd(0)} \e^{-\frac{1}{2}\left(a_1x_1^2+\cdots+a_nx_n^2\right)}\,\ud x }\\
  &< \frac{\e^{-\frac{1}{2}(a_1-\eps)(|z|-\delta)^2}}{\e^{-\frac{1}{2}(a_1-\eps)\gd^2}}\frac{  \int_{B^\gd(z)} \e^{-\frac{1}{2}\left(\eps x_1^2+(a_2-a_1+\eps)x_2^2+\cdots+(a_n-a_1+\eps)x_n^2\right)}\,\ud x }
                          { \int_{B^\gd(0)} \e^{-\frac{1}{2}(\eps x_1^2+(a_2-a_1+\eps)x_2^2+\cdots+(a_n-a_1+\eps)x_n^2)}\,\ud x }\\
  &< c\, \e^{-\frac{1}{2}(a_1-\eps)(|z|-\delta)^2}\frac{  \int_{B^\gd(z)}\hat\mu_0(\ud x) }
                          { \int_{B^\gd(0)} \hat\mu_0(\ud x) },
  \end{align*}
  for any $\eps<a_1$ and where $\hat\mu_0$ is a centred Gaussian
  measure on $\R^n$ with the Covariance matrix
  $\diag[1/\eps,1/(a_2-a_1+\eps),\cdots,1/(a_n-a_1+\eps)]$ (noting
  that $a_n\ge a_{n-1}\ge \cdots\ge a_1$).  By Anderson's inequality for the infinite dimensional spaces
  (see Theorem 2.8.10 of \cite{Bog98}) we have
  $\hat\mu_0(B(z,\gd))\le \hat\mu_0(B(0,\gd))$ and therefore
  \begin{equation*}
    \frac{J^\gd_{0,n}(z)}{J^\gd_{0,n}(0)}< c\,\e^{-\frac{1}{2}(a_1-\eps)(|z|-\delta)^2}
  \end{equation*}
  and since $\eps$ is arbitrarily small the result follows for the
  finite-dimensional case.

  To show the result for an infinite dimensional separable Banach
  space $X$, we first note that $\{e_j\}_{j=1}^\infty$, the orthogonal
  basis in the Cameron-Martin space of $X$ for $\mu_0$, separates the
  points in $X$, therefore $T:u\to \{e_j(u)\}_{j=1}^\infty$ is an
  injective map from $X$ into $\R^\infty$. Let $u_j=e_j(u)$ and
  \begin{equation*}
    P_nu=(u_1,u_2,\cdots,u_n,0,0,\cdots).
  \end{equation*}
  Then, since $\mu_0$ is a Radon measure, for the balls $B(0,\delta)$
  and $B(z,\delta)$, for any $\eps_0>0$, there exists large
  enough $N$ such that the cylindrical sets
  $A_0=P_n^{-1}(P_n(B^\gd(0))$ and $A_z=P_n^{-1}(P_n(B^\gd(z))$
  satisfy $\mu_0(B^\gd(0)\bigtriangleup A_0)<\eps_0$ and
  $\mu_0(B^\gd(z)\bigtriangleup A_z)<\eps_0$ for $n>N$
  \cite{Bog98}, where $\bigtriangleup$ denotes the symmetric difference.
  Let $z_j=(z,e_j)$ and
  $z^n=(z_1,z_2,\cdots,z_n,0,\cdots)$ and for $0<\eps_1<\gd/2$,
  $n>N$ large enough so that $\|z-z^n\|_X\le \eps_1$. With
  $\alpha=c\,\e^{-\frac{a_1}{2}(\|z\|_X-\eps_1-\delta)^2}$ we
  have
  \begin{align*}
  J_0^\gd(z)
  &\le J_{0,n}^\gd(z^n)+\eps_0\\
  &\le \alpha J_{0,n}^\gd(0)+\eps_0\\
  &\le\alpha J_0^\gd(0)+(1+\alpha)\eps_0.
  \end{align*}
  Since $\eps_0$ and $\eps_1$ converge to zero as
  $n\to\infty$, the result follows.
\end{proof}

\begin{lemma}\label{l:limg0-nE}
  Suppose that $\bar z\not\in E$, $\{z^\gd\}_{\gd>0}\subset X$ and
  $z^\gd$ converges weakly to $z$ in $X$ as $\gd\to 0$.  Then for any
  $\eps>0$ there exists $\delta$ small enough such that
  \begin{equation*}
    \frac{J^\gd_0(z^\gd)}{J^\gd_0(0)}<\eps.
  \end{equation*}
\end{lemma}

\begin{proof}
  Let $\cC$ be the covariance operator of $\mu_0$, and $\{
  e_j\}_{j\in\N}$ the eigenfunctions of $\cC$ scaled with respect to
  the inner product of $E$, the Cameron-Martin space of $\mu_0$, so
  that $\{e_j\}_{j\in\N}$ forms an orthonormal basis in $E$. Let
  $\{\lambda_j\}$ be the corresponding eigenvalues and $a_j=1/\lambda_j$.
  Since  $z^\gd$ converges weakly to $\bar z$ in $X$ as $\gd\to 0$,
  \begin{equation}\label{e:wcX}
  e_j(z^\gd)\to e_j(\bar z),\quad \mbox{for any }j\in\N
  \end{equation}
  and as $\bar z\not\in E$, for any $A>0$, there exists $N$
  sufficiently large and $\tilde\gd>0$ sufficiently small such that
  \begin{equation*}
    \inf_{z\in B^{\tilde\gd}(\bz)}\left\{\sum_{j=1}^N a_j x_j^2\right\}>A^2.
  \end{equation*}
  where $x_j= e_j(z)$. By (\ref{e:wcX}), for $\delta_1<\tilde\gd$
  small enough we have $B^{\gd_1}(z^{\gd_1})\subset
  B^{\tilde\gd}(\bz)$ and therefore
  \begin{equation}\label{e:lA}
    \inf_{z\in B^{\gd_1}(z^{\gd_1})}\left\{\sum_{j=1}^N a_j x_j^2\right\}>A^2.
  \end{equation}
  Let $T_n:X\to\R^n$ map $z$ to $(e_1(z), \ldots,e_n(z))$, and consider
  $J_{0,n}^\gd(z)$ to be defined as in (\ref{e:defJ0n}).  Having
  (\ref{e:lA}), and choosing $\gd\le \gd_1$ such that
  $\e^{-\frac{1}{4}(a_1+\dots+a_N)\gd^2}>1/2$, for any $n\ge N$ we can
  write
  \begin{align*}
  \frac{J^\gd_{0,n}(T_nz^\delta)}{J^\gd_{0,n}(0)}
  &=\frac{ \int_{B^\gd(T_nz^\gd)} \e^{-\frac{1}{2}(a_1x_1^2+\cdots+a_nx_n^2)}\,\ud x }{ \int_{B^\gd(0)} \e^{-\frac{1}{2}\left(a_1x_1^2+\cdots+a_nx_n^2\right)}\,\ud x }\\
  &\le \frac{ \int_{B^\gd(T_nz^\gd)}
         \e^{-\frac{1}{4}(a_1x_1^2+\cdots+a_Nx_N^2)}\e^{-\frac{1}{2}(\frac{a_1}{2}x_1^2+\cdots+\frac{a_N}{2}x_N^2+a_{N+1}x_{N+1}^2\cdots+a_nx_n^2)}\,\ud x }
         { \int_{B^\gd(0)} \e^{-\frac{1}{4}(a_1x_1^2+\cdots+a_Nx_N^2)}\e^{-\frac{1}{2}(\frac{a_1}{2}x_1^2+\cdots+\frac{a_N}{2}x_N^2+a_{N+1}x_{N+1}^2\cdots+a_nx_n^2)}\,\ud x }\\
  &\le \frac{ \e^{-\frac{1}{4}A^2}\int_{B^\gd(T_nz^\gd)} \e^{-\frac{1}{2}(\frac{a_1}{2}x_1^2+\cdots+\frac{a_N}{2}x_N^2+a_{N+1}x_{N+1}^2\cdots+a_nx_n^2)}\,\ud x }
         { \frac{1}{2}\int_{B^\gd(0)} \e^{-\frac{1}{2}(\frac{a_1}{2}x_1^2+\cdots+\frac{a_N}{2}x_N^2+a_{N+1}x_{N+1}^2\cdots+a_nx_n^2)}\,\ud x }\\
  &\le 2\e^{-\frac{1}{4}A^2}.
  \end{align*}
  As $A>0$ was arbitrary, the constant in the last line of the above
  equation can be made arbitrarily small, by making $\delta$
  sufficiently small and $n$ sufficiently large.  Having this and
  arguing in a similar way to the final paragraph of proof of
  Lemma~\ref{l:limg0}, the result follows.
\end{proof}

\begin{corollary}\label{c:limg0-nE}
Suppose that $z\not\in E$. Then
\begin{equation*}
  \lim_{\gd\to 0}\frac{J_0^\gd(z)}{J_0^\gd(0)}=0.
\end{equation*}
\end{corollary}

\begin{lemma}\label{l:limg0-wc}
  Consider $\{z^\gd\}_{\gd>0}\subset X$ and suppose that $z^\gd$
  converges weakly and not strongly to $0$ in $X$ as $\gd\to 0$.  Then
  for any $\eps>0$ there exists $\delta$ small enough such that
  \begin{equation*}
    \frac{J^\gd_0(z^\gd)}{J^\gd_0(0)}<\eps.
  \end{equation*}
\end{lemma}

\begin{proof}
  Since $z^\gd$ converges weakly and not strongly to $0$, we have
  \begin{equation*}
    \liminf_{\delta\to 0}\|z^\delta\|_X>0
  \end{equation*}
  and therefore for $\gd_1$ small enough there exists $\alpha>0$ such
  that $\|z^\gd\|_X>\alpha$ for any $\gd<\gd_1$.  Let $\lambda_j$, $a_j$
  and $e_j$, $j\in\N$, be defined as in the proof of Lemma
  \ref{l:limg0-nE}.  Since $z^\gd\rightharpoonup 0$ as $\gd\to 0$,
  \begin{equation}\label{e:wcX0}
  e_j(z^\gd)\to 0,\quad \mbox{for any }j\in\N
  \end{equation}
  Also, as for $\mu_0$-almost every $x\in X$, $x=\sum_{j\in\N}
  e_j(x)\hat{e_j}$ and $\{\hat{e}_j=e_j/\sqrt{\lambda_j}\}$ is an
  orthonormal basis in $X^*_{\mu_0}$ (closure of $X^*$ in
  $L^2(\mu_0)$) \cite{Bog98}, we have
  \begin{equation}\label{e:L2b}
  \sum_{j\in\N} (e_j(x))^2<\infty\quad
  \mbox{for $\mu_0$-almost every $x\in X$}.
  \end{equation}
  Now, for any $A>0$, let $N$ large enough such that $a_N>A^2$. Then,
  having (\ref{e:wcX0}) and (\ref{e:L2b}), one can choose
  $\gd_2<\gd_1$ small enough and $N_1>N$ large enough so that for
  $\gd<\gd_2$ and $n>N_1$
  \begin{equation*}
    \sum_{j=1}^N(e_j(z^\delta))^2< \frac{C\alpha}{2},\quad
    \mbox{and}\quad \sum_{j=N+1}^n(e_j(z^\delta))^2> \frac{C\alpha}{2}.
  \end{equation*}
  Therefore, letting $J_{0,n}^\gd(z)$ and $T_n$ be defined as in the
  proof of Lemma~\ref{l:limg0-nE}, we can write
  \begin{align*}
  &\frac{J^\gd_{0,n}(T_nz^\delta)}{J^\gd_{0,n}(0)}\\
  &\quad=\frac{ \int_{B^\gd(T_nz^\gd)} \e^{-\frac{1}{2}(a_1x_1^2+\cdots+a_nx_n^2)}\,\ud x }{ \int_{B^\gd(0)} \e^{-\frac{1}{2}\left(a_1x_1^2+\cdots+a_nx_n^2\right)}\,\ud x }\\
  &\quad\le \frac{ \int_{B^\gd(T_nz^\gd)}
         \e^{-\frac{A^2}{2}(x_{N+1}^2+\cdots+x_n^2)}\e^{-\frac{1}{2}({a_1}x_1^2+\cdots+{a_N}x_N^2+(a_{N+1}-{A^2})x_{N+1}^2\cdots+(a_n-{A^2})x_n^2)}\,\ud x }
  { \int_{B^\gd(0)} \e^{-\frac{A^2}{2}(x_{N+1}^2+\cdots+x_n^2)}\e^{-\frac{1}{2}({a_1}x_1^2+\cdots+{a_N}x_N^2+(a_{N+1}-{A^2})x_{N+1}^2\cdots+(a_n-{A^2})x_n^2)}\,\ud x }
  \\
  &\quad\le \frac{ \e^{-\frac{1}{2}A^2(\frac{C\alpha}{2}-\gd^2)}\int_{B^\gd(T_nz^\gd)} \e^{-\frac{1}{2}(\frac{a_1}{2}x_1^2+\cdots+\frac{a_N}{2}x_N^2+a_{N+1}x_{N+1}^2\cdots+a_nx_n^2)}\,\ud x }
         { \e^{-\frac{1}{2}A^2\gd^2}\int_{B^\gd(0)} \e^{-\frac{1}{2}(\frac{a_1}{2}x_1^2+\cdots+\frac{a_N}{2}x_N^2+a_{N+1}x_{N+1}^2\cdots+a_nx_n^2)}\,\ud x }\\
  &\quad\le 2\e^{-\frac{C\alpha}{4}A^2},
  \end{align*}
  if $\gd<\gd_2$ is small enough so that $\e^{A\gd^2}<2$.  Having this
  and arguing in a similar way to the final paragraph of proof of
  Lemma~\ref{l:limg0}, the result follows.
\end{proof}

Having these preparations in place, we can give the proof of
Theorem~\ref{t:MAP}.

\begin{proof}({\em of Theorem~\ref{t:MAP}}) i) We first show
  $\{z^\delta\}$ is bounded in $X$.  By Assumption~\ref{a:asp1}.(ii)
  for any $r>0$ there exists $K=K(r)>0$ such that
  \begin{equation*}
    \Phi(u) \le K(r)
  \end{equation*}
  for any $u$ satisfying $\|u\|_X < r$; thus $K$ may be assumed to be
  a non-decreasing function of $r$.  This implies that
  \begin{equation*}
    \max_{z\in E}\int_{B^{\delta}(z)}\e^{-\Phi(u)}\,\mu_0(\ud u)
    \ge \int_{B^{\delta}(0)}\e^{-\Phi(u)}\,\mu_0(\ud u)
    \ge \e^{-K(\delta)} \int_{B^{\delta}(0)}\mu_0(\ud u).
  \end{equation*}
  We assume that $\delta\le 1$ and then the inequality above shows
  that
  \begin{equation}\label{e:zX1}
  \frac{J^\delta(z^\delta)}{\int_{B^{\delta}(0)}\mu_0(\ud u)}
  \ge \frac{1}{Z}\e^{-K(1)}=\eps_1
  \end{equation}
  noting that $\eps_1$ is independent of $\delta$.

  We also can write
  \begin{align*}
  ZJ^\delta(z)
  &=\int_{B^\delta(z)}\e^{-\Phi(u)}\mu_0(\ud u)\\
  &\le \e^{-M}\int_{B^\delta(z)}\mu_0(\ud u)\\
  &=: \e^{-M} J^\delta_0(z),
  \end{align*}
  which implies that for any $z\in X$ and $\delta>0$
  \begin{equation}\label{e:J0}
  J^\delta_0(z)\ge Z\e^M J^\delta(z)
  \end{equation}
  Now suppose $\{z^\delta\}$ is not bounded in $X$, so that for any
  $R>0$ there exists $\delta_R$ such that $\|z^{\gd_R}\|_X>R$ (with
  $\gd_R\to 0$ as $R\to\infty$).  By (\ref{e:J0}), (\ref{e:zX1}) and
  definition of $z^{\gd_R}$ we have
  \begin{equation*}
    J_0^{\gd_R}(z^{\gd_R})\ge Z\e^M J^{\gd_R}(z^{\gd_R})
    \ge Z\e^M J^{\gd_R}(0)
    \ge \e^{M}\e^{-K(1)}J_0^{\gd_R}(0)
  \end{equation*}
  implying that for any $\gd_R$ and corresponding $z^{\gd_R}$
  \begin{equation*}
    \frac{J_0^{\gd_R}(z^{\gd_R})}{J_0^{\gd_R}(0)}\ge c=\e^{M}\e^{-K(1)}.
  \end{equation*}
  This contradicts the result of Lemma~\ref{l:limg0} (below) for
  $\delta_R$ small enough.  Hence there exists $R,\delta_R>0$ such
  that
  \begin{equation*}
  \|z^\delta\|_X\le R,\quad\mbox{ for any }\delta<\delta_R.
  \end{equation*}
  Therefore there exists a $\bar z\in X$ and a subsequence of
  $\{z^\delta\}_{0<\delta<\delta_R}$ which converges weakly in $X$ to
  $\bar z\in X$ as $\delta\to 0$.


  Now, suppose either
  \begin{itemize}
  \item[a)] there is no strongly convergent subsequence of
    $\{z^\delta\}$ in $X$, or
  \item[b)] if there is one, its limit $\bar z$ is not in $E$.
  \end{itemize}

  Let $U_E=\{u\in E: \|u\|_E\le 1\}$. Each of the above situations
  imply that for any positive $A\in \R$, there is a $\delta^\dagger$
  such that for any $\delta\le \delta^\dagger$,
  \begin{equation}\label{e:fringe}
    B^\delta(z^\delta)\cap \left(B^\delta(0)+AU_E\right)=\emptyset.
  \end{equation}

  We first show that, $\bar z$ has to be in $E$.  By definition of
  $z^\delta$ we have (for $\delta<1$)
  \begin{equation}\label{e:zd-contra}
  1\le \frac{J^\delta(z^\delta)}{J^\delta(0)}\le \frac{\e^M}{\e^{-K(1)}}\,\frac{\int_{B^\delta(z^\delta)}\mu_0(\ud u)}{\int_{B^\delta(0)}\mu_0(\ud u)}
  \end{equation}
  Supposing $\bar z\not\in E$, in Lemma~\ref{l:limg0-nE} we show that
  for any $\eps>0$ there exists $\delta$ small enough such that
  \begin{equation*}
  \frac{\int_{B^\delta(z^\delta)}\mu_0(\ud u)}{\int_{B^\delta(0)}\mu_0(\ud u)}<\eps.
  \end{equation*}
  Hence choosing $A$ in (\ref{e:fringe}) such that
  $\e^{-A^2/2}<\frac{1}{2}\,\e^{K(1)}\e^{-M}$, and setting
  $\eps=\e^{-A^2/2}$, from (\ref{e:zd-contra}), we get $1\le
  J^\delta(z^\delta)/J^\delta(0)<1$ which is a contradiction. We
  therefore have $\bar z\in E$.

  Now, knowing that $\bz\in E$, we can show that the $z^\delta$
  converges strongly in $X$. Suppose not.  Then for $z^\gd-\bz$ the
  hypotheses of Lemma~\ref{l:limg0-wc} are satisfied. Again choosing
  $A$ in (\ref{e:fringe}) such that
  $\e^{-A^2/2}<\frac{1}{2}\,\e^{K(1)}\e^{-M}$, and setting
  $\eps=\e^{-A^2/2}$, from Lemma~\ref{l:limg0-wc} and
  (\ref{e:zd-contra}), we get $1\le J^\delta(z^\delta)/J^\delta(0)<1$
  which is a contradiction.  Hence there is a subsequence of
  $\{z^\delta\}$ converging strongly in $X$ to $\bar z\in E$.


  ii) Let $z^*=\argmin I(z) \in E$; existence is assured by
  Theorem~\ref{t:OM}.  By Assumption~\ref{a:asp1} (iii) we have
  \begin{equation*}
  \frac{J^\delta(z^\delta)}{J^\delta(\bar z)}
  \le \e^{-\Phi(z^\delta)+\Phi(\bar z)}\e^{(L_1+L_2)\delta}
     \frac{\int_{B^\delta(z^\delta)}\mu_0(\ud u)}{\int_{B^\delta(\bar z)}\mu_0(\ud u)}
  \end{equation*}
  with $L_1=L(\|z^\delta\|_X+\delta)$ and $L_2=L(\|\bar
  z\|_X+\delta)$. Therefore, since $\Phi$ is continuous on $X$ and
  $z^\delta\to \bar z$ in X,
  \begin{equation*}
  \limsup_{\delta\to 0}\frac{J^\delta(z^\delta)}{J^\delta(\bar z)}
  \le\limsup_{\delta\to 0} \frac{\int_{B^\delta(z^\delta)}\mu_0(\ud u)}{\int_{B^\delta(\bar z)}\mu_0(\ud u)}
  \end{equation*}
  Suppose $\{z^\delta\}$ is not bounded in $E$ or if it is, it only
  converges weakly (and not strongly) in $E$. Then $\|\bar z\|_E <
  \liminf_{\delta\to 0} \|z^\delta\|_E$ and hence for small enough
  $\delta$, $\|\bar z\|_E<\|z^\delta\|_E$. Therefore for the centered
  Gaussian measure $\mu_0$, since $\|z^\delta-\bar z\|_X\to 0$ we have
  \begin{equation*}
    \limsup_{\delta\to 0}\frac{\int_{B^\delta(z^\delta)}\mu_0(\ud u)}{\int_{B^\delta(\bar z)}\mu_0(\ud u)}\le 1.
  \end{equation*}
  This, since by definition of $z^\delta$, $J^\delta(z^\delta)\ge
  J^\delta(\bar z)$ and hence
  \begin{equation*}
    \liminf_{\delta\to 0}\bigl(J^\delta(z^\delta)/ J^\delta(\bar z)\bigr)\ge 1,
  \end{equation*}
  implies that
  \begin{equation}\label{e:bzsup}
    \lim_{\delta\to 0}\frac{J^\delta(z^\delta)}{J^\delta(\bar z)}=1.
  \end{equation}
  In the case where $\{z^\delta\}$ converges strongly to $\bar z$ in
  $E$, by the Cameron-Martin Theorem we have
  \begin{equation*}
    \frac{\int_{B^\delta(z^\delta)}\mu_0(\ud u)}{\int_{B^\delta(\bar z)}\mu_0(\ud u)}
    =\frac{ \e^{-\frac{1}{2}\|z^\delta\|_E^2} \int_{B^\delta(0)}\e^{\la z^\delta,u \ra_E}\mu_0(\ud u) }{ \e^{-\frac{1}{2}\|\bar z\|_E^2} \int_{B^\delta(0)}\e^{\la \bar z,u \ra_E}\mu_0(\ud u) }
  \end{equation*}
  and then by an argument very similar to the proof of Theorem~18.3 of
  \cite{Lif95} one can show that
  \begin{equation*}
    \lim_{\delta\to 0}\frac{\int_{B^\delta(z^\delta)}\mu_0(\ud u)}{\int_{B^\delta(\bar z)}\mu_0(\ud u)}=1
  \end{equation*}
  and (\ref{e:bzsup}) follows again in a similar way. Therefore $\bar
  z$ is a MAP estimator of measure $\mu$.

  It remains to show that $\bar z$ is a minimiser of $I$. Suppose
  $\bar z$ is not a minimiser of $I$ so that $I(\bar z)-I(z^*)>0$.
  Let $\delta_1$ be small enough so that in the equation before
  (\ref{e:Jlims}) $1<r_1(\delta)< \e^{I(\bar z)-I(z^*)}$ for any
  $\delta<\delta_1$ and therefore
  \begin{equation}\label{e:bzmin}
    \frac{J^\delta(\bar z)}{J^\delta(z^*)}
    \le r_1(\gd) \e^{-I(\bar z)+I(z^*)}<1.
  \end{equation}
  Let $\alpha=r_1(\gd) \e^{-I(\bar z)+I(z^*)}$. We have
  \begin{equation*}
    \frac{J^\delta(z^\delta)}{J^\delta(z^*)}
    =\frac{J^\delta(z^\delta)} {J^\delta(\bar z)} \frac{J^\delta(\bar z)} {J^\delta(z^*)}
  \end{equation*}
  and this by (\ref{e:bzmin}) and (\ref{e:bzsup}) implies that
  \begin{equation*}
    \limsup_{\delta\to 0} \frac{J^\delta(z^\delta)}{J^\delta(z^*)}
    \le \,\alpha\,\limsup_{\delta\to 0} \frac{J^\delta(z^\delta)}{J^\delta(\bar z)}< 1,
  \end{equation*}
  which is a contradiction, since by definition of $z^\delta$,
  $J^\delta(z^\delta)\ge J^\delta(z^*)$ for any $\delta>0$.
\end{proof}

\begin{corollary}\label{c:MAPmin}
  Under the conditions of Theorem~\ref{t:MAP} we have the following:
  \begin{itemize}
  \item[i)] Any MAP estimator, given by Definition~\ref{d:MAP},
    minimises the Onsager-Machlup functional $I$.
  \item[ii)] Any $z^*\in E$ which minimises the Onsager-Machlup
    functional $I$, is a MAP estimator for measure $\mu$ given by
    \eqref{e:radon}.
  \end{itemize}
\end{corollary}

\begin{proof}
\begin{itemize}
\item[i)] Let $\tilde z$ be a MAP estimator. By Theorem~\ref{t:MAP} we
  know that $\{z^\gd\}$ has a subsequence which strongly converges in
  $X$ to $\bz$.  Let $\{z^\alpha\}$ be the said subsequence. Then by
  (\ref{e:bzsup}) one can show that
  \begin{equation*}
    \lim_{\delta\to 0}\frac{J^\delta(z^\delta)}{J^\delta(\bar z)}=\lim_{\alpha\to 0}\frac{J^\alpha(z^\alpha)}{J^\alpha(\bar z)}=1.
  \end{equation*}
  By the above equation and since $\tilde z$ is a MAP estimator, we can write
  \begin{equation*}
    \lim_{\delta\to 0}\frac{J^\delta(\tilde z)}{J^\delta(\bz)}
    = \lim_{\gd\to 0}\frac{J^\delta(z^\delta)}{J^\delta(\bz)}  \lim_{\gd\to 0}\frac{J^\delta(\tilde z)}{J^\delta(z^\delta)}
    =1.
  \end{equation*}
  Then Corollary~\ref{c:limg0-nE} implies that $\tilde z\in E$, and
  supposing that $\tilde z$ is not a minimiser of $I$ would result in
  a contradiction using an argument similar to last paragraph of the
  proof of the above theorem.

\item[ii)] Note that the assumptions of Theorem~\ref{t:MAP} imply
  those of Theorem~\ref{t:OM}. Since $\bz$ is a minimiser of $I$ as
  well, by Theorem~\ref{t:OM} we have
  \begin{equation*}
    \lim_{\gd\to 0}\frac{J^\gd (\bz)}{J^\gd(z^*)}=1.
  \end{equation*}
  Then we can write
  \begin{equation*}
    \lim_{\delta\to 0}\frac{J^\delta( z^*)}{J^\delta(z^\gd)}
    = \lim_{\gd\to 0}\frac{J^\delta(\bz)}{J^\delta(z^\delta)}
    = \lim_{\gd\to 0}\frac{J^\delta(z^*)}{J^\delta(\bz)}
    =1.
  \end{equation*}
  The result follows by Definition~\ref{d:MAP}.
\end{itemize}
\end{proof}


\section{Bayesian Inversion and Posterior Consistency}
\label{s:consistency}

The structure \eqref{e:radon}, where $\mu_0$ is Gaussian, arises in
the application of the Bayesian methodology to the solution of inverse
problems. In that context it is interesting to study {\em posterior
  consistency}: the idea that the posterior concentrates near the
truth which gave rise to the data, in the small noise or large sample size
limits; these two limits are intimately related and indeed there are
theorems that quantify this connection for certain linear inverse
problems \cite{BrLo96}.

In this section we describe the Bayesian
approach to nonlinear inverse problems, as outlined in the introduction.
We assume that the data is found from application of
$G$ to the truth $\utr$ with additional noise:
\begin{equation*}
  y=G(\utr)+\zeta.
\end{equation*}
The posterior distribution $\mu^y$ is then of the
form \eqref{e:radon1} and
in this case it is convenient to extend the Onsager-Machlup
functional $I$ to a mapping from $X \times \bbR^J$ to $\R$, defined as
\begin{equation*}
  I(u;y)
  = \Phi(u;y)+\frac{1}{2}\|u\|^2_E.
\end{equation*}

We study posterior consistency of MAP estimators in both the small noise and large sample size
limits.  The corresponding results are presented in
Theorems \ref{t:J_n} and~\ref{t:J},
respectively.  Specifically we characterize the sense in which the MAP
estimators concentrate on the truth underlying the data in the small
noise and large sample size limits.

\subsection{Large Sample Size Limit}

Let us denote the exact solution by $\utr$ and suppose that as data we
have the following $n$ random vectors
\begin{equation*}
  y_j
  = \cG(\utr)+\eta_j, \quad j=1, \ldots,n
\end{equation*}
with $y_j\in\R^K$ and $\eta_j\sim\cN(0,\cC_1)$ independent identically
distributed random variables.  Thus, in the general setting, we have
$J=nK$, $G(\cdot)=\bigl(\cG(\cdot),\cdots,\cG(\cdot)\bigr)$ and
$\Sigma$ a block diagonal matrix with $\cC_1$ in each block.  We have
$n$ independent observations each polluted by ${\cal O}(1)$ noise, and
we study the limit $n \to \infty$.  Corresponding to this set of data
and given the prior measure $\mu_0\sim\cN(0,\cC_0)$ we have the
following formula for the posterior measure on $u$:
\begin{equation*}
  \frac{\ud \mu^{y_1, \ldots,y_n}}{\ud \mu_0}(u)
  \propto\exp\left( -\frac{1}{2}\sum_{j=1}^n|y_j-\cG(u)|_{\cC_1}^2 \right).
\end{equation*}
Here, and in the following, we use the notation
$\left<\cdot,\cdot\right>_{\cC_1}=\left<\cC_1^{-1/2}\cdot,\cC_1^{-1/2}\cdot\right>$,
and $|\cdot|_{\cC_1}^2=\left<\cdot,\cdot\right>_{\cC_1}$: By Corollary
\ref{c:MAPmin} MAP estimators for this problem are minimisers of
\begin{equation}\label{eq:2I}
I_n:=\|u\|_E^2+\sum_{j=1}^n|y_j-\cG(u)|_{\cC_1}^2.
\end{equation}
Our interest is in studying properties of the limits of minimisers
$u_n$ of $I_n$, namely the MAP estimators corresponding to the
preceding family of posterior measures.  We have the following theorem
concerning the behaviour of $u_n$ when $n\to\infty$.

\begin{theorem}\label{t:J}
  Assume that $\cG\colon X\to\R^K$ is Lipschitz on bounded sets and
  $u^\dagger \in E$.  For every $n \in \bbN$, let $u_n\in E$ be a
  minimiser of $I_n$ given by \eqref{eq:2I}.  Then there exists a
  $u^*\in E$ and a subsequence of $\{u_n\}_{n \in \bbN}$ that
  converges weakly to $u^*$ in $E$, almost surely.  For any such $u^*$
  we have $\cG(u^*)=\cG(\utr)$.
\end{theorem}

We describe some preliminary calculations useful in the proof of this
theorem, then give Lemma~\ref{l:J_n}, also useful in the proof, and
finally give the proof itself.

We first observe that, under the assumption that $\cG$ is Lipschitz on
bounded sets, Assumptions~\ref{a:asp1} hold for $\Phi$.  We note that
\begin{align*}
  I_n
  &=\|u\|_E^2+\sum_{j=1}^n|y_j-\cG(u)|_{\cC_1}^2\\
  &=\|u\|_E^2+n|\cG(\utr)-\cG(u)|_{\cC_1}^2
    +2\sum_{j=1}^n\langle \cG(\utr)-\cG(u), \cC_1^{-1} \eta_j\rangle.
\end{align*}
Hence
\begin{equation*}
  \argmin_u I_n=\argmin_u\left\{ \|u\|_E^2+n|\cG(\utr)-\cG(u)|_{\cC_1}^2
  +2\sum_{j=1}^n\langle \cG(\utr)-\cG(u), \cC_1^{-1} \eta_j\rangle  \right\}.
\end{equation*}
Define
\begin{equation*}
  J_n(u)=|\cG(\utr)-\cG(u)|_{\cC_1}^2
  +\frac{1}{n}\|u\|_E^2
  +\frac{2}{n}\sum_{j=1}^n \langle \cG(\utr)-\cG(u),
  \cC_1^{-1}\eta_j \rangle.
\end{equation*}
We have
\begin{equation*}
  \argmin_u I_n = \argmin_u J_n.
\end{equation*}

\begin{lem}\label{l:J_n}
  Assume that $\cG\colon X\to\R^K$ is Lipschitz on bounded sets.  Then
  for fixed $n\in\bbN$ and almost surely, there exists $u_n\in E$ such
  that
  \begin{equation*}
    J_n(u_n)
    = \inf_{u\in E}J_n(u).
  \end{equation*}
\end{lem}

\begin{proof}
  We first observe that, under the assumption that $\cG$ is Lipschitz
  on bounded sets and because for a given $n$ and fixed realisations
  $\eta_1, \ldots,\eta_n$ there exists an $r>0$ such that
  $\max\{|y_1|, \ldots,|y_n|\}<r$, Assumptions~\ref{a:asp1} hold for
  $\Phi$.  Since $\argmin_u I_n= \argmin_u J_n$ the result follows
  by Proposition~\ref{p:prop}.
\end{proof}

We may now prove the posterior consistency theorem.  From
\eqref{e:Eun_bd} onwards the proof is an adaptation of the proof of
Theorem 2 of \cite{BisH04}.
We note that, the assumptions on limiting behaviour of measurement
noise in \cite{BisH04} are stronger: property (9) of \cite{BisH04}
is not assumed here for our $J_n$.  On the other hand a frequentist
approach is used in \cite{BisH04}, while here since $J_n$ is coming
from a Bayesian approach, the norm in the regularisation term is
stronger (it is related to the Cameron-Martin space of the Gaussian
prior). That is why in our case asking what if $\utr$ is not in $E$
and only in $X$, is relevant and is answered in Corollary \ref{c:utrX}
below.

\begin{proof}\,\,{\em (of Theorem \ref{t:J})}
  By definition of $u_n$ we have
 \begin{equation*}
    |\cG(\utr)-\cG(u_n)|_{\cC_1}^2
    +\frac{1}{n}\|u_n\|_E^2
    +\frac{2}{n}\sum_{j=1}^n \langle \cG(\utr)-\cG(u_n),
\cC_1^{-1}\eta_j \rangle
    \,\le\,  \frac{1}{n}\|\utr\|_E^2.
  \end{equation*}
  Therefore
  \begin{equation*}
    |\cG(\utr)-\cG(u_n)|_{\cC_1}^2 +\frac{1}{n}\|u_n\|_E^2
    \,\le\,  \frac{1}{n}\|\utr\|_E^2
    +\frac{2}{n}\,|\cG(\utr)-\cG(u_n)|_{\cC_1}\,|\sum_{j=1}^n
\cC_1^{-1/2}\eta_j|.
  \end{equation*}
  Using Young's inequality (see Lemma 1.8 of \cite{book:Robinson2001}, for example) for the last term in the right-hand side we get
  \begin{equation*}
    \frac{1}{2}|\cG(\utr)-\cG(u_n)|_{\cC_1}^2 +\frac{1}{n}\|u_n\|_E^2
    \,\le\,  \frac{1}{n}\|\utr\|_E^2
    +\frac{2}{n^2}\,\bigl(\sum_{j=1}^n \cC_1^{-1/2}\eta_j\bigr)^2.
  \end{equation*}
  Taking expectation and noting that the $\{\eta_j\}$ are independent,
  we obtain
  \begin{equation*}
    \frac{1}{2}\bbE|\cG(\utr)-\cG(u_n)|_{\cC_1}^2 +\frac{1}{n}\bbE\|u_n\|_E^2
    \,\le\,  \frac{1}{n}\|\utr\|_E^2 +\frac{2K}{n}
  \end{equation*}
  where $K=\bbE|\cC_1^{-1/2}\eta_1|^2$.  This implies that
  \begin{equation}\label{e:EGlim}
    \bbE|\cG(\utr)-\cG(u_n)|_{\cC_1}^2\to 0\quad\mbox{as}\quad n\to\infty
  \end{equation}
  and
  \begin{equation}\label{e:Eun_bd}
    \bbE\|u_n\|_E^2\le \|\utr\|_E^2+2K.
  \end{equation}
  1) We first show using \eqref{e:Eun_bd} that there exist $u^*\in E$
  and a subsequence $\{u_{n_k(k)}\}_{k\in\bbN}$ of $\{u_n\}$ such that
  \begin{equation}\label{e:weakEu0}
    \bbE \la u_{n_k(k)},v\ra_E\to\bbE \la u^*,v\ra_E,\quad\mbox{for any }v\in E.
  \end{equation}
  Let $\{\phi_j\}_{j\in\bbN}$ be a complete orthonormal system for $E$. Then
  \begin{equation*}
    \bbE\la u_n,\phi_1 \ra_E
    \,\le\, \bbE\|u_n\|_E\,\|\phi_1\|_E
    \,\le\, \|\utr\|_E^2+2K.
  \end{equation*}
  Therefore there exists $\xi_1\in\R$ and a subsequence
  $\{u_{n_1(k)}\}_{k\in\bbN}$ of $\{u_n\}_{n\in\bbN}$, such that
  $\bbE\la u_{n_1(k)},\phi_1 \ra\to\xi_1$.  Now considering $\bbE\la
  u_{n_1(k)},\phi_2\ra$ and using the same argument we conclude that
  there exists $\xi_2\in\R$ and a subsequence
  $\{u_{n_2(k)}\}_{k\in\bbN}$ of $\{u_{n_1(k)}\}_{k\in\bbN}$ such that
  $\bbE\la u_{n_2(k)},\phi_2 \ra\to\xi_2$.  Continuing similarly we
  can show that there exist $\{\xi_j\}\in\R^\infty$ and
  $\{u_{n_1(k)}\}_{k\in\bbN}\supset \{u_{n_2(k)}\}_{k\in\bbN}
  \supset\dots\supset \{u_{n_j(k)}\}_{k\in\bbN}$ such that $\bbE\la
  u_{n_j(k)},\phi_j \ra\to\xi_j$ for any $j\in\bbN$ and as
  $k\to\infty$.  Therefore
  \begin{equation*}
    \bbE\la u_{n_k(k)},\phi_j \ra_E\to\xi_j,\quad\mbox{as }k\to\infty\mbox{ for any }j\in\bbN.
  \end{equation*}
  We need to show that $\{\xi_j\}\in \ell^2(\R)$. We have, for any $N\in \bbN$,
  \begin{equation*}
    \sum_{j=1}^N \xi_j^2
    \le \lim_{k\to\infty}\bbE \sum_{j=1}^N\la u_{n_k(k)},\phi_j \ra_E^2
    \le \limsup_{k\to\infty}\bbE\|u_{n_k(k)}\|_E^2\le \|\utr\|_E^2+2K.
  \end{equation*}
  Therefore $\{\xi_j\}\in \ell^2(\R)$ and $u^*:=\sum_{j=1}^\infty
  \xi_j\phi_j \in E$.  We can now write for any nonzero $v\in E$
  \begin{align*}
    \bbE\la u_{n_k(k)}&-u^*,v\ra_E
    =\bbE \sum_{j=1}^\infty \la v,\phi_j \ra_E\la u_{n_k(k)}-u^*,\phi_j\ra_E\\
    &\le N\|v\|_E\,\,\bbE \sup_{j\in\{1, \ldots,N\}}|\la u_{n_k(k)}-u^*,\phi_j\ra_E|
    +(\|\utr\|_E^2+2K)^{1/2}\sum_{j=N}^\infty |\la v,\phi_j \ra_E|
  \end{align*}
  Now for any fixed $\eps>0$ we choose $N$ large enough so that
  \begin{equation*}
    (\|\utr\|_E^2+2K)^{1/2}
    \sum_{j=N}^\infty |\la v,\phi_j \ra_E|< \frac12 \eps
  \end{equation*}
  and then $k$ large enough so that
  \begin{equation*}
  N\|v\|_E \, \bbE|\la u_{n_k(k)}-u^*,\phi_j\ra_E| < \frac12 \eps
    \quad\mbox{for any }1\le j\le N.
  \end{equation*}
  This demonstrates that $\bbE\la u_{n_k(k)}-u^*,v\ra_E\to 0$ as
  $k\to\infty$.

  2) Now we show almost sure existence of a convergent subsequence of
  $\{u_{n_k(k)}\}$.  By \eqref{e:EGlim} we have
  $|\cG(u_{n_k(k)})-\cG(\utr)|_{\cC_1}\to 0$ in probability as
  $k\to\infty$.  Therefore there exists a subsequence $\{u_{m(k)}\}$
  of $\{u_{n_k(k)}\}$ such that
  \begin{equation*}
    \cG(u_{m(k)})\to\cG(\utr)\quad\mbox{a.s. }\mbox{ as }k\to\infty.
  \end{equation*}
  Now by \eqref{e:weakEu0} we have $\la u_{m(k)}-u^*,v\ra_E\to 0$ in
  probability as $k\to\infty$ and hence there exists a subsequence
  $\{u_{\hat m(k)}\}$ of $\{u_{m(k)}\}$ such that $ u_{\hat m(k)}$ converges
  weakly to $u^*$ in $E$ almost surely as $k\to\infty$.
   Since $E$ is compactly embedded in $X$, this implies that $u_{\hat
    m(k)}\to u^*$ in $X$ almost surely as $k\to\infty$.  The result
  now follows by continuity of $\cG$.
\end{proof}

In the case that $\utr\in X$ (and not necessarily in $E$), we have the
following weaker result:

\begin{corollary}\label{c:utrX}
  Suppose that $\cG$ and $u_n$ satisfy the assumptions of Theorem
  \ref{t:J}, and that $\utr\in X$.  Then there exists a subsequence of
  $\{\cG(u_n)\}_{n\in\N}$ converging to $\cG(\utr)$ almost surely.
\end{corollary}

\begin{proof}
  For any $\eps>0$, by density of $E$ in $X$, there exists $v\in E$
  such that $\|\utr-v\|_X\le \eps$.  Then by definition of $u_n$ we
  can write
  \begin{align*}
      |\cG(\utr)-\cG(u_n)|_{\cC_1}^2
  &    +\frac{1}{n}\|u_n\|_E^2
      +\frac{2}{n}\sum_{j=1}^n \langle \cG(\utr)-\cG(u_n),\cC_1^{-1}\eta_j \rangle\\
  &    \,\le\,   |\cG(\utr)-\cG(v)|_{\cC_1}^2+\frac{1}{n}\|v\|_E^2+\frac{2}{n}\sum_{j=1}^n \langle \cG(\utr)-\cG(v),\cC_1^{-1}\eta_j \rangle.
  \end{align*}
  Therefore, dropping $\frac{1}{n}\|u_n\|_E^2$ in the left-hand
  side, and using Young's inequality we get
  \begin{equation*}
     \frac{1}{2} |\cG(\utr)-\cG(u_n)|_{\cC_1}^2
         \le\,  2|\cG(\utr)-\cG(v)|_{\cC_1}^2+\frac{1}{n}\|v\|_E^2+\frac{3}{n^2}\sum_{j=1}^n  |\cC_1^{-1/2}\eta_j|^2.
  \end{equation*}
  By local Lipschitz continuity of $\cG$,
  $|\cG(\utr)-\cG(v)|_{\cC_1}\le C\eps^2$, and therefore taking the
  expectations and noting the independence of $\{\eta_j\}$ we get
  \begin{equation*}
    \bbE |\cG(\utr)-\cG(u_n)|_{\cC_1}^2\le 4C\eps^2+\frac{2C_\eps}{n}+\frac{6K}{n},
  \end{equation*}
  implying that
  \begin{equation*}
    \limsup_{n\to\infty}\bbE|\cG(\utr)-\cG(u_n)|_{\cC_1}^2\le 4C\eps^2.
  \end{equation*}
  Since the $\lim\inf$ is obviously positive and $\eps$ was arbitrary,
  we have $\lim_{n\to\infty}\bbE|\cG(\utr)-\cG(u_n)|_{\cC_1}^2=0$.
  This implies that $|\cG(\utr)-\cG(u_n)|_{\cC_1}\to 0$ in
  probability. Therefore there exists a subsequence of $\{\cG(u_n)\}$
  which converges to $\cG(\utr)$ almost surely.
\end{proof}

\subsection{Small Noise Limit}
\label{ssec:sn}

Consider the case where as data we have the random vector
\begin{equation}\label{e:obs_del}
  y_n = \cG(\utr)+\frac{1}{n}\eta_n,
\end{equation}
for $n \in \bbN$ and with $\utr$ again as the true solution and
$\eta_j\sim \cN(0,\cC_1)$, $j\in\N$, Gaussian random vectors in
$\bbR^K$. Thus, in the preceding general setting, we have $G=\cG$ and
$J=K$.  Rather than having $n$ independent observations, we have an
observation noise scaled by small $\gamma=1/n$ converging to zero.
For this data and given the prior measure $\mu_0$ on $u$, we have the
following formula for the posterior measure:
\begin{equation*}
  \frac{\ud \mu^{y_n}}{\ud \mu_0}(u)
  \propto\exp\left( -\frac{n^2}{2}\left| y_n-\cG(u) \right|_{\cC_1}^2 \right).
\end{equation*}
By the result of the previous section, the MAP estimators for the
above measure are the minimisers of
\begin{equation}\label{eq:1I}
  I_n(u):=\|u\|_E^2+n^2|y_n-\cG(u)|_{\cC_1}^2.
\end{equation}
Our interest is in studying properties of the limits of minimisers of
$I_n$ as $n \to \infty$.  We have the following almost sure
convergence result.

\begin{theorem}\label{t:J_n}
  Assume that $\cG\colon X\to\R^K$ is Lipschitz on bounded sets, and
  $u^\dagger \in E$.  For every $n \in \bbN$, let $u_n\in E$ be a
  minimiser of $I_n(u)$ given by \eqref{eq:1I}.  Then there exists a
  $u^*\in E$ and a subsequence of $\{u_n\}_{n \in \bbN}$ that
  converges weakly to $u^*$ in $E$, almost surely.  For any such $u^*$
  we have $\cG(u^*)=\cG(\utr)$.
\end{theorem}

\begin{proof}
  The proof is very similar to that of Theorem~\ref{t:J} and
  so we only sketch differences.  We have
  \begin{align*}
  I_n
  &=\|u\|_E^2+n^2|y_n-\cG(u)|_{\cC_1}^2\\
  &=\|u\|_E^2+n^2|\cG(\utr)+\frac{1}{n}\eta_n-\cG(u)|_{\cC_1}^2\\
  &=\|u\|_E^2+n^2|\cG(\utr)-\cG(u)|_{\cC_1}^2+|\eta_n|_{\cC_1}^2
      +2\,n\left< \cG(\utr)-\cG(u),\eta_n \right>_{\cC_1}.
  \end{align*}
  Letting
  \begin{equation*}
    J_n(u)=\frac{1}{n^2}\|u\|_E^2+|\cG(\utr)-\cG(u)|_{\cC_1}^2
    +\frac{2}{n}\left< \cG(\utr)-\cG(u),\eta_n \right>_{\cC_1},
  \end{equation*}
  we hence have $\argmin_u I_n=\argmin_u J_n$.  For this $J_n$ the
  result of Lemma~\ref{l:J_n} holds true, using an argument similar to
  the large sample size case.  The result of Theorem~\ref{t:J_n} carries over
  as well. Indeed, by definition of $u_n$, we have
  \begin{equation*}
   |\cG(\utr)-\cG(u_n)|_{\cC_1}^2
    +\frac{1}{n^2}\|u_n\|_E^2
    +\frac{2}{n}\langle \cG(\utr)-\cG(u_n), \cC_1^{-1}\eta_n \rangle
    \,\le\,  \frac{1}{n^2}\|\utr\|_E^2.
  \end{equation*}
  Therefore
  \begin{equation*}
  |\cG(\utr)-\cG(u_n)|_{\cC_1}^2 +\frac{1}{n^2}\|u_n\|_E^2
    \,\le\,  \frac{1}{n^2}\|\utr\|_E^2
    +\frac{2}{n}\,|\cG(\utr)-\cG(u_n)|_{\cC_1}\,|\cC_1^{-1/2}\eta_n|.
  \end{equation*}
  Using Young's inequality for the last term in the right-hand side we get
  \begin{equation*}
    \frac{1}{2}|\cG(\utr)-\cG(u_n)|_{\cC_1}^2 +\frac{1}{n^2}\|u_n\|_E^2
    \,\le\,  \frac{1}{n^2}\|\utr\|_E^2
    +\frac{2}{n^2}\,|\cC_1^{-1/2}\eta_n|^2.
  \end{equation*}
  Taking expectation we obtain
  \begin{equation*}
    \bbE|\cG(\utr)-\cG(u_n)|_{\cC_1}^2 +\frac{1}{n^2}\bbE\|u_n\|_E^2
    \,\le\,  \frac{1}{n^2}\|\utr\|_E^2 +\frac{2K}{n^2}.
  \end{equation*}
  This implies that
  \begin{equation}\label{e:EGlim0}
    \bbE|\cG(\utr)-\cG(u_n)|_{\cC_1}^2\to 0\quad\mbox{as}\quad n\to\infty
  \end{equation}
  and
  \begin{equation}\label{e:Eun_bd0}
    \bbE\|u_n\|_E^2\le \|\utr\|_E^2+2K.
  \end{equation}
  Having (\ref{e:EGlim0}) and (\ref{e:Eun_bd0}), and with the same
  argument as the proof of Theorem~\ref{t:J}, it follows that there
  exists a $u^*\in E$ and a subsequence of $\{u_n\}$ that converges
  weakly to $u^*$ in $E$ almost surely, and for any such $u^*$ we have
  $\cG(u^*)=\cG(\utr)$.
\end{proof}

As in the large sample size case, here also if we have $\utr\in X$ and we do
not restrict the true solution to be in the Cameron-Martin space $E$,
one can prove, in a similar way to the argument of the proof of
Corollary~\ref{c:utrX}, the following weaker convergence result:

\begin{corollary}\label{c:utrX0}
  Suppose that $\cG$ and $u_n$ satisfy the assumptions of Theorem
  \ref{t:J_n}, and that $\utr\in X$.  Then there exists a subsequence
  of $\{\cG(u_n)\}_{n\in\N}$ converging to $\cG(\utr)$ almost surely.
\end{corollary}


\section{Applications in Fluid Mechanics}
\label{sec:fm}

In this section we present an application of the methods presented
above to filtering and smoothing in fluid dynamics, which is relevant
to data assimilation applications in oceanography and meteorology.  We
link the MAP estimators introduced in this paper to the variational
methods used in applications \cite{ben}, and we demonstrate posterior
consistency in this context.

We consider the 2D Navier-Stokes equation on the torus $\TT^{2} :=
[-1,1) \times [-1,1)$ with periodic boundary conditions:
\begin{eqnarray*}
  \begin{array}{cccc}
    \pd_tv - \nu \Delta v + v \cdot \nabla v + \nabla p &=& f
      & \mbox{for all $(x, t) \in \TT^{2} \times (0, \infty)$,} \\
    \nabla \cdot v &=& 0
      & \mbox{for all $(x, t) \in \TT^{2} \times (0, \infty)$,} \\
    v &=& u
      & \mbox{for all $(x,t) \in \TT^{2} \times \{0\}$.}
  \end{array}
\end{eqnarray*}
Here $v \colon \TT^{2} \times (0, \infty) \to \R^{2}$ is a
time-dependent vector field representing the velocity, $p \colon
\TT^{2} \times (0,\infty) \to \R$ is a time-dependent scalar field
representing the pressure, $f \colon \TT^{2} \to \R^{2}$ is a vector
field representing the forcing (which we assume to be time-independent
for simplicity), and $\nu$ is the viscosity.  We are interested in the
inverse problem of determining the initial velocity field $u$ from
pointwise measurements of the velocity field at later times.  This is
a model for the situation in weather forecasting where observations of
the atmosphere are used to improve the initial condition used for
forecasting.  For simplicity we assume that the initial velocity field
is divergence-free and integrates to zero over $\TT^2$, noting that
this property will be preserved in time.

Define
\begin{equation*}
  \cH
  := \left\{ \mbox{trigonometric polynomials }
    u\colon \TT^2 \to {\mathbb R}^2\,\Bigl|\, \nabla \cdot u = 0, \,\int_{\TT^{2}} u(x) \, \rd x = 0 \right\}
\end{equation*}
and $H$ as the closure of $\cH$ with respect to the
$(L^{2}(\TT^{2}))^{2}$ norm. We define the map $P\colon (L^{2}(\TT^{2}))^{2}
\to H$ to be the Leray-Helmholtz orthogonal projector (see
\cite{book:Robinson2001}).  Given $k = (k_1, k_2)^{\mathrm{T}}$,
define $k^{\perp} := (k_2, -k_1)^{\mathrm{T}}$.  Then an orthonormal
basis for $H$ is given by $\psi_k \colon \R^{2} \to \R^{2}$, where
\begin{equation*}
  \psi_k (x) := \frac{k^{\perp}}{|k|} \exp\Bigl( \pi i k \cdot
  x \Bigr)
\end{equation*}
for $k \in \Z^{2} \setminus \{0\}$.  Thus for $u \in H$ we may write
\begin{equation*}
  u = \sum_{k \in \Z^{2} \setminus \{0\}} u_k(t) \psi_k(x)
\end{equation*}
where, since $u$ is a real-valued function, we have the reality
constraint $u_{-k} = - \bar{u}_k$.  Using the Fourier decomposition
of $u$, we define the fractional Sobolev spaces
\begin{equation*}
  H^{s}
  := \Bigl\{ u \in H \Bigm| \sum_{k \in \Z^{2} \setminus \{0\}} (\pi^{2}\abs{k}^{2})^{s}\abs{u_k}^{2} < \infty \Bigr\}
\end{equation*}
with the norm $\norm{u}_s:= \bigl(\sum_k (\pi^{2}\abs{k}^{2})^{s}
\abs{u_k}^{2}\bigr)^{1/2}$, where $s \in \R$.  If $A=-P\Delta$, the
Stokes' operator, then $H^s=D(A^{s/2})$.  We assume that $f \in H^s$
for some $s>0$.

Let $t_\ell=\ell h$, for $\ell=0, \ldots, L$, and define $v_\ell \in
\bbR^M$ be the set of pointwise values of the velocity field given by
$\{v(x_m,t_\ell)\}_{m \in \bbK}$ where $\bbK$ is some finite set of
point in $\TT^2$ with cardinality $M/2$.  Note that each $v_\ell$
depends on $u$ and we may define $\cG_\ell\colon H \to \bbR^{M}$ by
$\cG_\ell(u)=v_\ell$.  We let $\{ {\eta}_\ell \}_{\ell \in
  \{1,\ldots,L\}}$ be a set of random variables in $\bbR^{M}$ which
perturbs the points $\{v_\ell \}_{\ell \in \{1,\ldots,L\}}$ to
generate the observations $\{ y_\ell \}_{\ell \in \{1,\ldots,L\}}$ in
$\bbR^{M}$ given by
\begin{equation*}
  y_{\ell}
  := v_{\ell} + \gamma {\eta}_{\ell}, \quad \ell \in \{1,\ldots, L\}.
\end{equation*}
We let $y=\{y_\ell\}_{\ell=1}^L$, the accumulated data up to time
$T=Lh$, with similar notation for $\eta$, and define $\cG\colon H \to
\bbR^{ML}$ by $\cG(u)= \bigl(\cG_1(u), \ldots,\cG_L(u)\bigr)$.  We now
solve the inverse problem of finding $u$ from $y=\cG(u)+\gamma \eta$.
We assume that the prior distribution on $u$ is a Gaussian
$\mu_0=N(0,\cC_0)$, with the property that $\mu_0(H)=1$ and that the
observational noise $\{\eta_\ell\}_{\ell \in \{1,\ldots,L\}}$ is
i.i.d.\ in $\bbR^M$, independent of $u$, with $\eta_1$ distributed
according to a Gaussian measure $N(0, I)$.  If we define
\begin{equation*}
  \Phi(u)=\frac{1}{2\gamma^2}\sum_{j=1}^L |y_j-\cG_j(u)|^2
\end{equation*}
then under the preceding assumptions the Bayesian inverse problem for
the posterior measure $\mu^y$ for $u|y$ is well-defined and is
Lipschitz in $y$ with respect to the Hellinger metric (see
\cite{CDRS09}).  The Onsager-Machlup functional in this case is given
by
\begin{equation*}
  I_{\mathrm{NS}}(u)
  = \frac{1}{2}\|u\|_{\cC_0}^2+\Phi(u).
\end{equation*}
We are in the setting of subsection~\ref{ssec:sn}, with $\gamma=1/n$
and $K=ML$.  In the applied literature approaches to assimilating data
into mathematical models based on minimising $I_{\mathrm{NS}}$ are
known as {\em variational methods}, and sometimes as 4DVAR \cite{ben}.

\begin{figure}
  \begin{center}
    \includegraphics[width=0.8\textwidth,trim=0 125 0 140,clip]{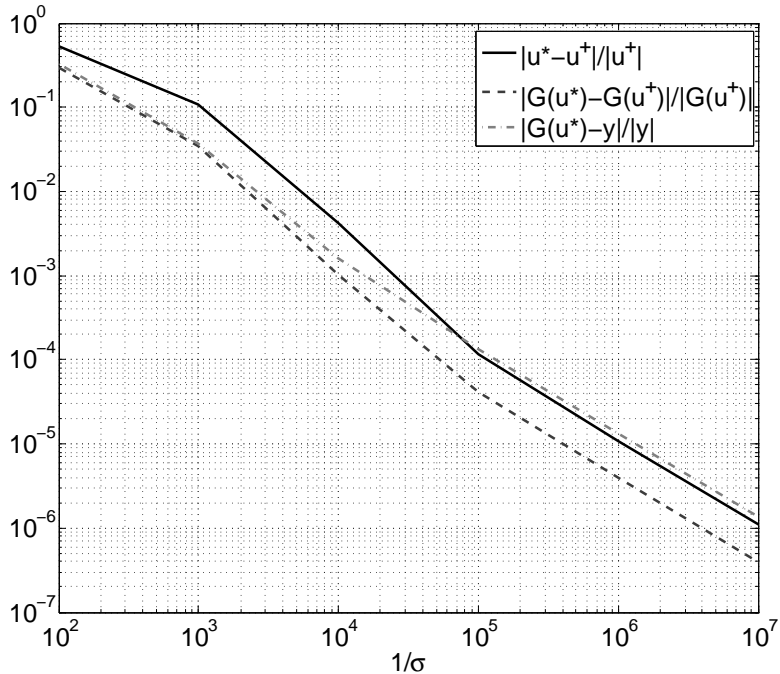}
  \end{center}
  \caption{\label{fig:post_cons}Illustration of posterior consistency in
    the fluid mechanics application.  The three curves given are the
    relative error of the MAP estimator $u^*$ in reproducing the
    truth, $u^{\dagger}$ (solid), the relative error of the map
    $\cG(u^*)$ in reproducing $\cG(u^{\dagger})$ (dashed), and the
    relative error of $\cG(u^*)$ with respect to the observations $y$
    (dash-dotted).}
\end{figure}

We now describe numerical experiments concerned with studying
posterior consistency in the case $\gamma \to 0$.  We let $\cC_0 =
A^{-2}$ noting that if $u \sim \mu_0$, then $u \in H^s$ almost surely
for all $s<1$; in particular $u \in H$.  Thus $\mu_0(H)=1$ as
required.  The forcing in $f$ is taken to be $f=\nabla^{\perp}\Psi$,
where $\Psi=\cos(\pi k \cdot x)$ and $\nabla^{\perp}=J\nabla$ with $J$
the canonical skew-symmetric matrix, and $k=(5,5)$.  The dimension of
the attractor is determined by the viscosity parameter $\nu$.  For the
particular forcing used there is an explicit steady state for all
$\nu>0$ and for $\nu \geq 0.035$ this solution is stable (see
\cite{majda2006non}, Chapter 2 for details).  As $\nu$ decreases the
flow becomes increasingly complex and we focus subsequent studies of
the inverse problem on the mildly chaotic regime which arises for $\nu
= 0.01$. We use a time-step of $\delta t = 0.005$.  The data is
generated by computing a true signal solving the Navier-Stokes
equation at the desired value of $\nu$, and then adding Gaussian
random noise to it at each observation time.  Furthermore, we let $h =
4 \delta t = 0.02$ and take $L=10$, so that $T=0.2$.  We take $M=32^2$
spatial observations at each observation time. The observations are
made at the gridpoints; thus the observations include all numerically
resolved, and hence observable, wavenumbers in the system.
Since the noise is added in spectral space in practice, for
convenience we define $\sigma = \gamma/ \sqrt{M}$ and present results
in terms of $\sigma$.  The same grid is used for computing
the reference solution and for computing the MAP estimator.

Figure~\ref{fig:post_cons} illustrates the posterior consistency which
arises as the observational noise strength $\gamma \to 0$.  The three
curves shown quantify: (i) the relative error of the MAP estimator
$u^*$ compared with the truth, $u^{\dagger}$; (ii) the relative error
of $\cG(u^*)$ compared with $\cG(u^{\dagger})$; and (iii) the relative
error of $\cG(u^*)$ with respect to the observations $y$.  The figure
clearly illustrates Theorem~\ref{t:J_n}, via the dashed curve for
(ii), and indeed shows that the map estimator itself is converging to
the true initial condition, via the solid curve (i), as $\gamma \to
0$.  Recall that the observations approach the true value of the
initial condition, mapped forward under $\cG$, as $\gamma \to 0$, and
note that the dashed and dashed-dotted curves shows that the image of
the MAP estimator under the forward operator $\cG$, $\cG(u^*)$, is
closer to $\cG(u^\dagger)$ than $y$, asymptotically as $\gamma \to 0$.


\section{Applications in Conditioned Diffusions}
\label{sec:cd}

In this section we consider the MAP estimator for conditioned
diffusions, including bridge diffusions and an application to
filtering/smoothing.  We identify the Onsager-Machlup functional
governing the MAP estimator in three different cases. We
demonstrate numerically that this functional may have more than one
minimiser.  Furthermore, we illustrate the results of the consistency
theory in section~\ref{s:consistency} using numerical experiments.
Subsection~\ref{ssec:un} concerns the unconditioned case, and includes
the assumptions made throughout. Subsections~\ref{ssec:bd} and
\ref{ssec:fs} describe bridge diffusions and the filtering/smoothing
problem respectively.  Finally, subsection~\ref{ssec:ne} is devoted to
numerical experiments for an example in filtering/smoothing.

\subsection{Unconditioned Case}
\label{ssec:un}

For simplicity we restrict ourselves to scalar processes with additive
noise, taking the form
\begin{equation}\label{eq:SDE}
  du
  = f(u) \,dt + \sigma \,dW, \quad u(0)=u^-.
\end{equation}
If we let $\nu$ denote the measure on $X := C\bigl([0,T];\R\bigr)$
generated by the stochastic differential equation (SDE) given
in~\eqref{eq:SDE}, and $\nu_0$ the same measure obtained in the case
$f \equiv 0$, then the Girsanov theorem states that $\nu \ll \nu_0$
with density
\begin{equation*}
  \frac{\ud \nu}{\ud \nu_0}(u)
  = \exp\Bigl(-\frac{1}{2\sigma^2}\int_0^T \bigl|f\bigl(u(t)\bigr)\bigr|^2 \,dt+\frac{1}{\sigma^2}\int_0^T f\bigl(u(t)\bigr) \,du(t)\Bigr).
\end{equation*}
If we choose an $F\colon \R \to \R$ with $F'(u)=f(u)$, then an
application of It\^o's formula gives
\begin{equation*}
  dF\bigl(u(t)\bigr)
  = f\bigl(u(t)\bigr) \,du(t)
  +\frac{\sigma^2}{2} f'\bigl(u(t)\bigr) \,dt,
\end{equation*}
and using this expression to remove the stochastic integral we obtain
\begin{equation}\label{eq:rnd2}
  \frac{\ud \nu}{\ud \nu_0}(u)
  \propto \exp\Bigl(-\frac{1}{2\sigma^2}\int_0^T \bigl(\bigl|f\bigl(u(t)\bigr)\bigr|^2+\sigma^2 f'\bigl(u(t)\bigr)\bigr) \,dt
  +\frac{1}{\sigma^2}F\bigl(u(T)\bigr)\Bigr).
\end{equation}
Thus, the measure $\nu$ has a density with respect to the Gaussian
measure $\nu_0$ and \eqref{eq:rnd2} takes the form~\eqref{e:radon}
with $\mu=\nu$ and $\mu_0=\nu_0$: we have
\begin{equation*}
  \frac{\ud \nu}{\ud \nu_0}(u)
  \propto \exp\bigl( - \Phi_1(u) \bigr)
\end{equation*}
where  $\Phi_1\colon X \to \R$ is defined by
\begin{equation}\label{eq:phi1}
  \Phi_1(u)
  = \int_0^T \Psi\bigl(u(t)\bigr) \,dt - \frac{1}{\sigma^2}F\bigl(u(T)\bigr)
\end{equation}
and
\begin{equation*}
  \Psi(u)=\frac{1}{2\sigma^2} \Bigl(|f(u)|^2+\sigma^2 f'(u) \Bigr).
\end{equation*}

We make the following assumption concerning the vector field $f$
driving the SDE:
\begin{assumption}
  The function $f = F'$ in~\eqref{eq:SDE} satisfies the following conditions.
  \begin{enumerate}
  \item $F \in C^2(\R,\R)$ for all $u \in \R$.
  \item There is $M \in \R$ such that
    $\Psi(u) \ge M$ for all $u \in \R$ and $F(u) \le M$
    for all $u \in \R$.
  \end{enumerate}
\end{assumption}

Under these assumptions, we see that $\Phi_1$ given by \eqref{eq:phi1}
satisfies Assumptions~\ref{a:asp1} and, indeed, the slightly stronger
assumptions made in Theorem~\ref{t:MAP}.  Let $H^1[0,T]$ denote the
space of absolutely continuous functions on $[0,T]$.  Then the
Cameron-Martin space $E_1$ for $\nu_0$ is
\begin{equation*}
  E_1=\Bigl\{v \in H^1[0,T]\Bigm| \int_0^T \bigl|v'(s)\bigr|^2\,ds<\infty
  \mbox{ and } v(0)=0\Bigr\}
\end{equation*}
and the Cameron-Martin norm is given by
\begin{equation*}
  \|v\|_{E_1}
  = \sigma^{-1}\|v\|_{H^1}
\end{equation*}
where
\begin{equation*}
  \|v\|_{H^1}
  = \Bigl(\int_0^T \bigl|v'(s)\bigr|^2\,ds\Bigr)^{\frac12}.
\end{equation*}

The mean of $\nu_0$ is the constant function $m \equiv u^-$ and so,
using Remark~\ref{r:mean}, we see that the Onsager-Machlup functional
for the unconditioned diffusion \eqref{eq:SDE} is thus $I_1\colon E_1
\to \R$ given by
\begin{equation*}
  I_1(u)
  = \Phi_1(u) + \frac{1}{2\sigma^2}\| u - u^- \|_{H^1}^2
  = \Phi_1(u) + \frac{1}{2\sigma^2}\| u \|_{H^1}^2.
\end{equation*}
Together, Theorems \ref{t:OM} and~\ref{t:MAP} tell us that this
functional attains its minimum over $E_1'$ defined by
\begin{equation*}
  E_1'
  = \Bigl\{v \in H^1[0,T]\Bigm| \int_0^T \bigl|v'(s)\bigr|^2\,ds<\infty
    \mbox{ and } v(0)=u^-\Bigr\}.
\end{equation*}
Furthermore such minimisers define MAP estimators for the
unconditioned diffusion~\eqref{eq:SDE}, \textit{i.e.}\ the most likely
paths of the diffusion.

We note that the regularity of minimisers for $I_1$ implies that the
MAP estimator is $C^2$, whilst sample paths of the SDE \eqref{eq:SDE}
are not even differentiable.  This is because the MAP estimator
defines the centre of a tube in $X$ which contains the most likely
paths. The centre itself is a smoother function than the paths.  This
is a generic feature of MAP estimators for measures defined via
density with respect to a Gaussian in infinite dimensions.

\subsection{Bridge Diffusions}
\label{ssec:bd}

In this subsection we study the probability measure generated by
solutions of \eqref{eq:SDE}, conditioned to hit $u^+$ at time $1$ so
that $u(T)=u^+$, and denote this measure $\mu$. Let $\mu_0$ denote the
Brownian bridge measure obtained in the case $f \equiv 0$.  By
applying the approach to determining bridge diffusion measures in
\cite{HSV07} we obtain, from \eqref{eq:rnd2}, the expression
\begin{equation}\label{eq:rnd3}
  \frac{\ud \mu}{\ud \mu_0}(u) \propto \exp\Bigl(-\int_0^T \Psi\bigl(u(t)\bigr) \,dt
  +\frac{1}{\sigma^2}F\bigl(u^+\bigr)\Bigr).
\end{equation}

Since $u^+$ is fixed we now define $\Phi_2\colon X \to \R$ by
\begin{equation*}
  \Phi_2(u)
  = \int_0^T \Psi\bigl(u(t)\bigr) \,dt
\end{equation*}
and then \eqref{eq:rnd3} takes again the form \eqref{e:radon}.  The
Cameron-Martin space for the (zero mean) Brownian bridge is
\begin{equation*}
  E_2
  = \Bigl\{v \in H^1[0,T]\Bigm| \int_0^T \bigl|v'(s)\bigr|^2\,ds<\infty
    \mbox{ and } v(0)=v(T)=0\Bigr\}
\end{equation*}
and the Cameron-Martin norm is again $\sigma^{-1}\|\quark\|_{H^1}$.
The Onsager-Machlup function for the unconditioned diffusion
\eqref{eq:SDE} is thus $I_2\colon E_2' \to \R$ given by
\begin{equation*}
  I_2(u)
  = \Phi_2(u) + \frac{1}{2\sigma^2}\| u - m\|_{H^1}^2
\end{equation*}
where $m$, given by $m(t) = \frac{T-t}{T} u^- + \frac{t}{T} u^+$ for all
$t\in[0,T]$ , is the mean of $\mu_0$ and
\begin{equation*}
  E_2'
  = \Bigl\{v \in H^1[0,T]\Bigm| \int_0^T \bigl|v'(s)\bigr|^2\,ds<\infty
    \mbox{ and } v(0)=u^-, u(T)=u^+ \Bigr\}.
\end{equation*}
The MAP estimators for~$\mu$ are found by minimising $I_2$
over~$E_2'$.

\subsection{Filtering and Smoothing}
\label{ssec:fs}

We now consider conditioning the measure $\nu$ on observations of the
process $u$ at discrete time points. Assume that we observe $y \in
\R^J$ given by
\begin{equation}\label{eq:times}
  y_j=u(t_j)+\eta_j
\end{equation}
where $0<t_1<\cdots<t_J<T$ and the $\eta_j$ are independent
identically distributed random variables with $\eta_j \sim
N(0,\gamma^2)$.  Let $\bbQ_0(\ud y)$ denote the $\R^J$-valued Gaussian
measure $N(0,\gamma^2 I)$ and let $\bbQ(\ud y|u)$ denote the
$\R^J$-valued Gaussian measure $N(\cG u,\gamma^2 I)$ where $\cG \colon
X \to \R^J$ is defined by
\begin{equation*}
  \cG u=\bigl(u(t_1),\cdots,u(t_J)\bigr).
\end{equation*}
Recall $\nu_0$ and $\nu$ from the unconditioned case and define the
measures $\bbP_0$ and $\bbP$ on $X\times\R^J$ as follows. The measure
$\bbP_0(\ud u,\ud y)=\nu_0(\ud u)\bbQ_0(\ud y)$ is defined to be an
independent product of $\nu_0$ and $\bbQ_0$, whilst $\bbP(\ud u,\ud
y)=\nu(\ud u)\bbQ(\ud y|u)$. Then
\begin{equation*}
  \frac{\ud \bbP}{\ud \bbP_0}(u,y)
  \propto \exp\Bigl(-\int_0^T \Psi\bigl(u(t)\bigr) \,dt
  +\frac{1}{\sigma^2}F\bigl(u(T)\bigr)-\frac{1}{2\gamma^2}\sum_{j=1}^J |y_j-u(t_j)|^2\Bigr)
\end{equation*}
with constant of proportionality depending only on $y$.  Clearly, by
continuity,
\begin{equation*}
  \inf_{\|u\|_X\le 1} \exp\Bigl(-\int_0^T \Psi\bigl(u(t)\bigr) \,dt
  +\frac{1}{\sigma^2}F\bigl(u(T)\bigr)-\frac{1}{2\gamma^2}\sum_{j=1}^J |y_j-u(t_j)|^2\Bigr)>0
\end{equation*}
and hence
\begin{equation*}
  \int_{\|u\|_X \le 1} \exp\Bigl(-\int_0^T \Psi\bigl(u(t)\bigr) \,dt
  +\frac{1}{\sigma^2}F\bigl(u(T)\bigr)-\frac{1}{2\gamma^2}\sum_{j=1}^J |y_j-u(t_j)|^2\Bigr)
  \,\nu_0(du)>0.
\end{equation*}
Applying the conditioning Lemma 5.3 in \cite{HSV07} then gives
\begin{equation*}
  \frac{\ud \mu^y}{\ud \nu_0}(u) \propto
  \exp\Bigl(-\int_0^T \Psi\bigl(u(t)\bigr) \,dt
  +\frac{1}{\sigma^2}F\bigl(u(T)\bigr)-\frac{1}{2\gamma^2}\sum_{j=1}^J |y_j-u(t_j)|^2\Bigr).
\end{equation*}
Thus we define
\begin{equation*}
  \Phi_3(u)
  =\int_0^T \Psi\bigl(u(t)\bigr) \,dt
  -\frac{1}{\sigma^2}F\bigl(u(T)\bigr)
  +\frac{1}{2\gamma^2}\sum_{j=1}^J |y_j-u(t_j)|^2.
\end{equation*}
The Cameron-Martin space is again $E_1$ and the Onsager-Machlup
functional is thus $I_3\colon E_1' \to \R$, given by
\begin{equation}\label{eq:I3}
  I_3(u) = \Phi_3(u) + \frac{1}{2\sigma^2}\|u\|_{H^1}^2.
\end{equation}
The MAP estimator for this setup is, again, found by minimising the
Onsager-Machlup functional~$I_3$.

The only difference between the potentials $\Phi_1$ and $\Phi_3$, and
thus between the functionals $I_1$ for the unconditioned case and
$I_3$ for the case with discrete observations, is the presence of the
term $\frac{1}{2\gamma^2}\sum_{j=1}^J |y_j-u(t_j)|^2$.  In the
Euler-Lagrange equations describing the minima of~$I_3$, this term
leads to Dirac distributions at the observation points $t_1, \ldots,
t_J$ and it transpires that, as a consequence, minimisers of $I_3$
have jumps in their first derivates at $t_1, \ldots, t_J$.  This
effect can be clearly seen in the local minima
of~$I_3$ shown in figure~\ref{fig:smooth-minima}.

\subsection{Numerical Experiments}
\label{ssec:ne}

In this section we perform three numerical experiments related to the
MAP estimator for the filtering/smoothing problem presented in
section~\ref{ssec:fs}.

\begin{figure}
  \begin{center}
    \includegraphics{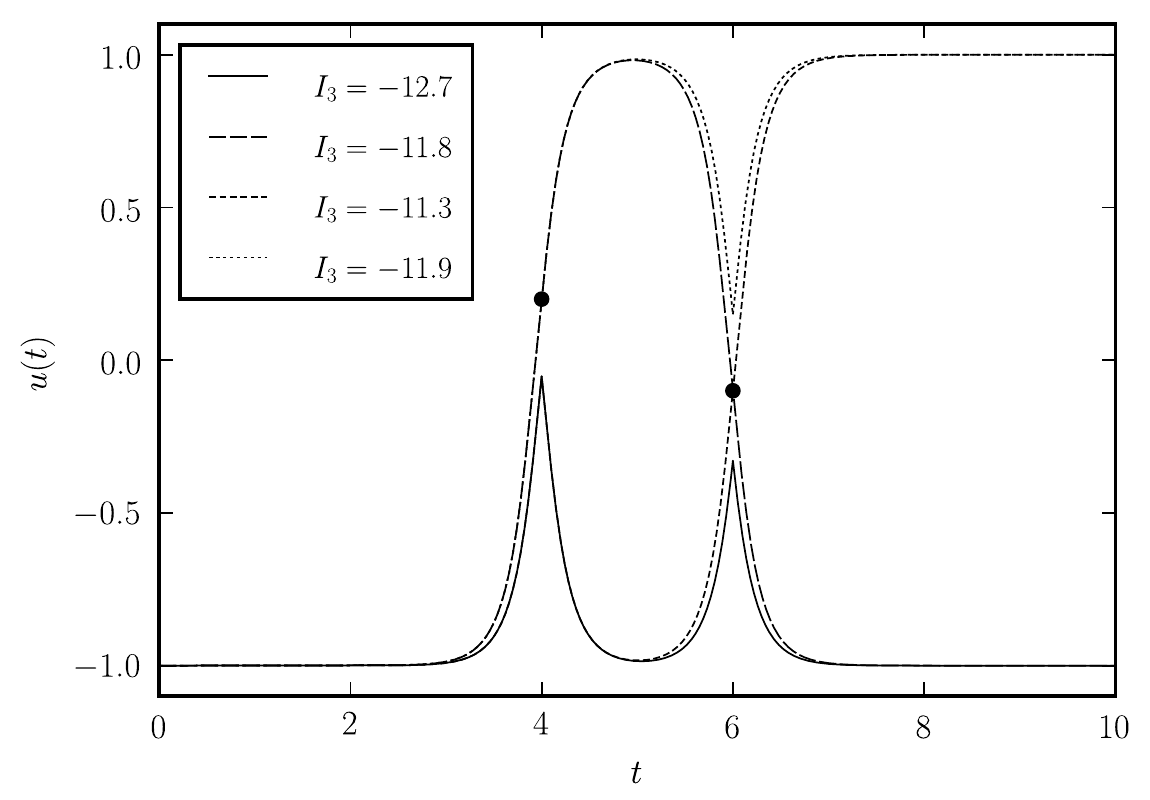}
  \end{center}
  \caption{\label{fig:smooth-minima}Illustration of the problem of
    local minima of $I$ for the smoothing problem with a small number
    of observations.  The process $u(t)$ starts at $u(0)=-1$ and moves
    in a double-well potential with stable equilibrium points at $-1$
    and $+1$.  Two observations of the process are indicated by the
    two black circles.  The curves correspond to four different local
    minima of the functional $I_3$ for this situation.}
\end{figure}

For the experiments we generate a random ``signal'' by numerically
solving the SDE~\eqref{eq:SDE}, using the Euler-Maruyama method, for a
double-well potential~$F$ given by
\begin{equation*}
  F(u) = - \frac{(1-u)^2(1+u)^2}{1+u^2},
\end{equation*}
with diffusion constant $\sigma=1$ and initial value $u^- = -1$.  From
the resulting solution $u(t)$ we generate random observations $y_1,
\ldots, y_J$ using~\eqref{eq:times}.  Then we implement the
Onsager-Machlup functional~$I_3$ from equation~\eqref{eq:I3} and use
numerical minimisation, employing the Broyden-Fletcher-Goldfarb-Shanno
method (see~\cite{Fle00}; we use the implementation found in the GNU
scientific library~\cite{GSL}), to find the minima of $I_3$.  The same
grid is used for numerically solving the SDE and for approximating the
values of~$I_3$.

The first experiment concerns the problem of local minima of $I_3$.
For small number of observations we find multiple local minima; the
minimisation procedure can converge to different {\em local} minima,
depending on the starting point of the optimisation.  This effect
makes it difficult to find the MAP estimator, which is the {\em
  global} minimum of $I_3$, numerically.  The problem is illustrated
in figure~\ref{fig:smooth-minima}, which shows four different local
minima for the case of $J=2$ observations.  In the presence of
local minima, some care is needed when numerically computing
the MAP estimator.  For example, one could start the minimisation
procedure with a collection of different starting points, and take
the best of the resulting local minima as the result.
One would expect this
problem to become less pronounced as the number of observations
increases, since the observations will ``pull'' the MAP estimator
towards the correct solution, thus reducing the number of local
minima.  This effect is confirmed by experiments: for larger numbers
of observations our experiments found only one local minimum.

\begin{figure}
  \begin{center}
    \includegraphics{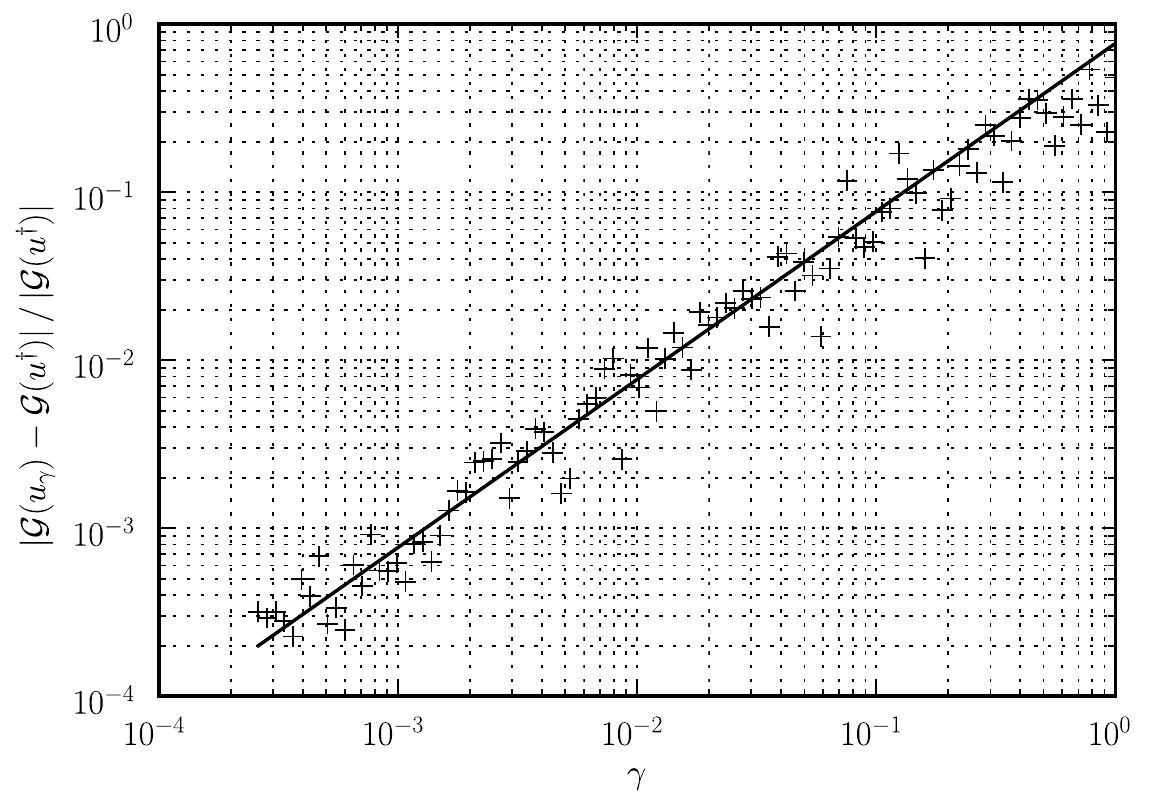}
  \end{center}
  \caption{\label{fig:smooth-noise}Illustration of posterior
    consistency for the smoothing problem in the small-noise limit.
    The marked points correspond the maximum-norm distance between the
    true signal $u^\dagger$ and the MAP estimator $u_\gamma$ with
    $J=5$ evenly spaced observations.  The map $\cG(u) =
    \bigl(u(t_1), \ldots, u(t_J)\bigr)$ is the projection of the
    path onto the observation points.  The solid line is a fitted
    curve of the form~$c \gamma$.}
\end{figure}

The second experiment concerns posterior consistency of the MAP
estimator in the small noise limit.  Here we use a fixed number~$J$ of
observations of a fixed path of~\eqref{eq:SDE}, but let the variance
$\gamma^2$ of the observational noise $\eta_j$ converge to~$0$.
Noting that the exact path of the SDE, denoted by $u^\dagger$
in~\eqref{e:obs_del}, has the regularity of a Brownian motion and
therefore the observed path is not contained in the Cameron-Martin
space~$E_3$, we are in the situation described in Corollary
\ref{c:utrX0}.  Our experiments indicate that we have $\cG(u_\gamma)
\to \cG(u^\dagger)$ as $\gamma \downarrow0$, where $u_\gamma$ denotes
the MAP estimator corresponding to observational variance~$\gamma^2$,
confirming the result of Corollary~\ref{c:utrX0}.
As discussed above, for small values of $\gamma$ one would
expect the minimum of $I_3$ to be unique and indeed experiments
where different starting points of the optimisation procedure were tried
did not find different minima for small~$\delta$.
The result of a
simulation with $J=5$ is shown in figure~\ref{fig:smooth-noise}.

\begin{figure}
  \begin{center}
    \includegraphics{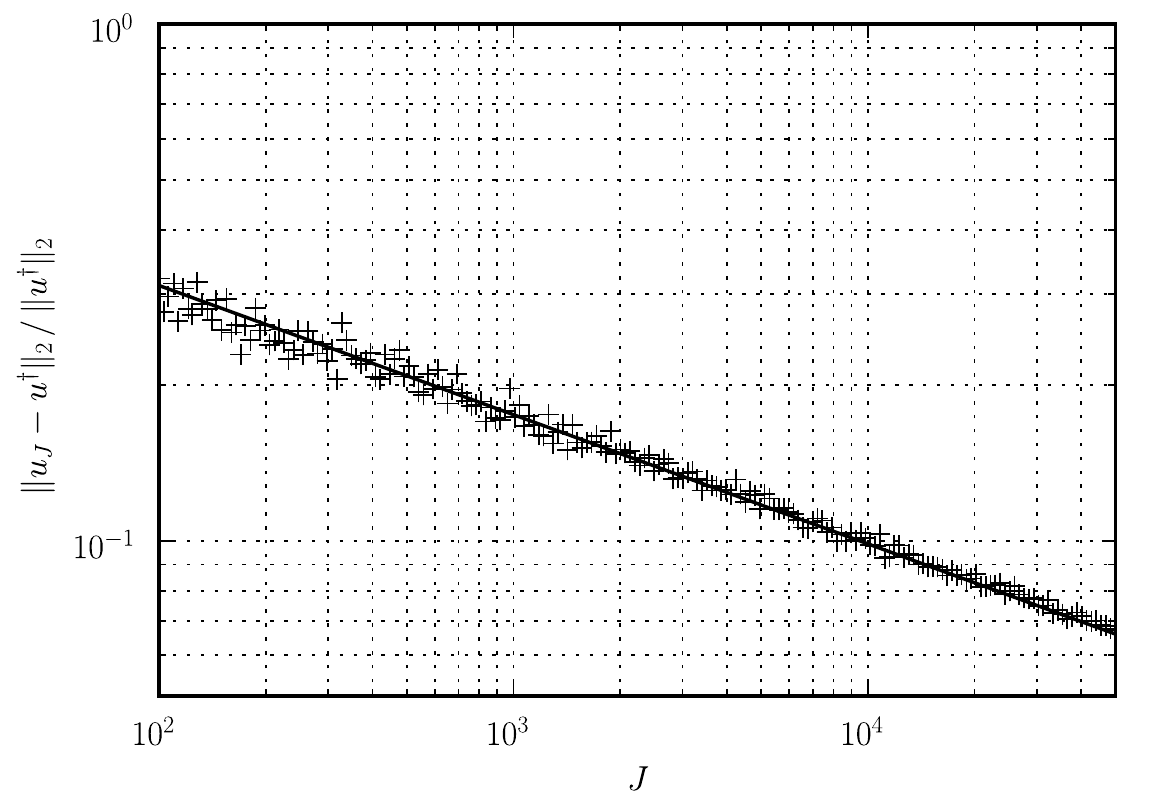}
  \end{center}
  \caption{\label{fig:smooth-data}Illustration of posterior
    consistency for the smoothing problem in the large sample size limit.
    The marked points correspond the supremum-norm distance between
    the true signal $u^*$ and the MAP estimator $u^\dagger_J$ with $J$
    evenly spaced observations.  The solid line give a fitted curve of
    the form $c J^{-\alpha}$; the exponent $\alpha=-1/4$
    was found numerically.}
\end{figure}

Finally, we perform an experiment to illustrate posterior consistency
in the large sample size limit: for this experiment we still use one fixed
path $u^\dagger$ of the SDE~\eqref{eq:SDE}.  Then, for different
values of~$J$, we generate observations $y_1, \ldots, y_J$
using~\eqref{eq:times} at equidistantly spaced times $t_1, \ldots,
t_J$, for fixed $\gamma=1$, and then determine the $L^2$ distance of
the resulting MAP estimate $u_J$ to the exact path~$u^\dagger$.
As discussed above, for large values of $J$ one would
expect the minimum of $I_3$ to be unique and indeed experiments
where different starting points of the optimisation procedure were tried
did not find different minima for large~$J$.  The
situation considered here is not covered by the theoretical results
from section~\ref{s:consistency}, but the results of the numerical
experiment, shown in figure~\ref{fig:smooth-data} indicate that
posterior consistency still holds.


\bibliographystyle{99}

\end{document}